\documentclass[12pt,reqno]{amsart}
 \usepackage{graphicx}
 \usepackage[utf8]{inputenc}
 \usepackage[T1]{fontenc}
 \usepackage{lmodern}
 \usepackage[normalem]{ulem}
 \usepackage{verbatim}
 \usepackage{bbm}
 \usepackage{stmaryrd}
 \usepackage{amsmath}
 \usepackage{amssymb}
 \usepackage{dsfont}
\usepackage{amsthm}

\newtheorem{theo}{Theorem}[section]%

\newtheorem{assum}{Assumption}[section]%
\newtheorem{cor}[theo]{Corollary}%
\newtheorem{lemma}[theo]{Lemma}%
\newtheorem{rem}[theo]{Remark}%
\newtheorem{prop}[theo]{Proposition}%
\newtheorem{Ex}[theo]{Example}%

\def\mun{{\hat \mu_N}}
\def\ra{{\rightarrow}}
\newcommand{\E}{\mathbb E}
\newcommand{\Var}{\mathbb{V}\text{ar}}
\newcommand{\Pp}{\mathbb P}
\newcommand{\C}{\mathbb C}
\newcommand{\R}{\mathbb R}
\newcommand{\N}{\mathbb N}

\def\la{{\lambda}}
\begin{document}

\title{Large deviations   for the largest eigenvalue of  Rademacher matrices}
\author{Alice~Guionnet}
\address[Alice Guionnet]{Universit\'e de Lyon, ENSL, CNRS,  France}
\email{Alice.Guionnet@umpa.ens-lyon.fr}

\author{Jonathan~Husson}
\address[Jonathan Husson]{Universit\'e de Lyon, ENSL, CNRS,  France}
\email{Jonathan.Husson@umpa.ens-lyon.fr}
\thanks{This work was supported in part by Labex MILYON}

\maketitle

\begin{abstract}
In this article, we consider random Wigner matrices, that is symmetric matrices such that the subdiagonal entries of $X_n$ are independent, centered, and with variance one except on the diagonal where the entries have variance two. We prove that, under some suitable hypotheses on the laws of the entries, the law of the largest eigenvalue satisfies a large deviation principle with the same rate function as in the Gaussian case. The crucial assumption is that the Laplace transform of the entries must be bounded above by  the Laplace transform of a centered  Gaussian variable with same variance. This is satisfied by  the Rademacher law and the uniform law on $[ - \sqrt{3}, \sqrt{3}]$. We extend our result to complex entries Wigner matrices and Wishart matrices.\end{abstract}

\section{Introduction}

Very few large deviation principles could be  proved so far   in random matrix theory. Indeed, the natural quantities of interest such as the spectrum and the eigenvectors are complicated functions of the entries. Hence,   even if  one considers the simplest model of  Wigner matrices which  are self-adjoint with  independent  identically distributed entries   above the diagonal, the probability that the empirical measure of the eigenvalues or  the largest eigenvalue deviates towards an unlikely value is very difficult to estimate. A well known  case where probabilities of large deviations can be estimated is the case where the entries are Gaussian, centered and well chosen covariances,  the so-called Gaussian ensembles. In this case, the joint law of the eigenvalues has an explicit form, independent of the eigenvectors, displaying  a strong Coulomb gas interaction. This formula could be used to prove a large deviations principle for the empirical measure in \cite{BAG97} and for the largest eigenvalue \cite{BADG01} (see also \cite{Majum} for further discussions of the Wishart case, and \cite{Majum2}). More recently, in a breakthrough paper,  C.Bordenave and P. Caputo \cite{BordCap}  tackled the case of matrices with heavy tails, that is  Wigner matrices with entries with stretched exponential tails, going to zero at infinity more slowly than a Gaussian tail. The driving idea to approach this question is to show that large deviations are in this case created by a few large entries, so that the empirical measure deviates towards the free convolution of the semi-circle law and the limiting spectral measure of the matrix created by these few large entries. This idea could be also used to grasp the large deviations of  the largest eigenvalue by F.Augeri \cite{fanny}.  In the Wishart case, \cite{fey} considered the large deviations for the largest eigenvalue of very thin Wishart matrices $W=GG^*$, in the regime where the matrix $G$ is $L\times M$ with $L$ much smaller than $M$. Hence large deviations for bounded entries, or simply entries with sub-Gaussian tails, remained mysterious in the case of Wigner matrices or Wishart matrices with $L$ of order $M$. In this article we analyze the large deviations of the largest eigenvalue of Wigner matrices with Rademacher or uniformly distributed random variables. More precisely our result holds for any  independent identically distributed entries with distribution with Laplace transform 
 bounded above by  the  Laplace transform  of the Gaussian law with the same variance. 
We then prove a large deviation principle with   the same  rate function  than in the Gaussian case: large deviations are universal in this class of measures. We show that this result generalizes 
to complex entries Wigner matrices as well as to Wishart matrices. We are considering the case of general sub-Gaussian entries in a companion paper with F. Augeri. We show in particular that the rate function is  different from the rate function of the Gaussian case, at least for deviations towards very large values. 

\vspace{0.9 cm}

\subsection{Statement of the results}

We consider a family of independent real random variables $(a_{i,j}^{(1)})_{0 \leq i\leq j \leq N}$, such that   the variables $a_{i,j}^{(1)}$ are distributed according to the  laws $\mu_{i,j}^N$.   We moreover assume that the $\mu^{N}_{i,j}$ are centered :
 $$\mu_{i,j}^{N}(x)=\int x d\mu^{N}_{i,j}(x)=0$$ 
 and with covariance:
$$\mu^N_{i,j}(x^{2})=\int x^{2 } d\mu_{{i,j}}^N(x)=1, \forall 1\le i<j\le N,\qquad \mu^N_{i,i}(x^{2})=2, \quad \forall 1\le i\le N\,.$$
We say that a probability measure $\mu$ has a sharp sub-Gaussian  Laplace transform  iff

\begin{equation}\label{boundL} \forall t \in \R, T_\mu(t) =\int \exp\{tx\} d\mu(x) \leq \exp\big\{\frac{t^2 \mu(x^{2}) }{ 2}\big \}\,. \end{equation}
The terminology ``sharp'' comes from the fact that for $t$ small, we must have
$$T_{\mu}(t)\ge  \exp\{\frac{t^2 \mu(x^{2}) }{ 2} (1+o(t))\}\,.$$

Then we assume that

\begin{assum}[A0]\label{A0} We assume that the $\mu_{i,j}^{N}$ satisfy a sharp Gaussian Laplace transform in the sense that
\begin{itemize} 
\item $(\mu_{i,j}^N)_{i\le j}$ have a sharp sub-Gaussian Laplace transform,
\item  The $\mu_{i,j}^{N}$ have  a uniform lower bounded Laplace transform: For any $\delta>0$ there exists $\varepsilon(\delta)>0$ such that for any $|t|\le\varepsilon(\delta)$, any $1\le i\le j\le N$,  any $N\in \N$,
$$T_{\mu_{i,j}^N}(t)\ge  \exp\{\frac{(1-\delta) t^2 \mu_{i,j}^N(x^{2}) }{ 2} \}\,.$$
\end{itemize}
Moreover, we assume that the $T_{\mu_{i,j}^N}$ are uniformly $C^3$ in a neighborhood of the origin: for $\epsilon>0$ small enough $\sup_{|t|\le \epsilon}\sup_{i,j,N} |\partial_t^3\ln T_{\mu_{i,j}^N}(t)|$ is finite.
 \end{assum}
 Observe that the $\mu_{i,j}^{N}$ have  a uniform lower bounded Laplace transform as soon as they do not depend on $N$ and there are finitely many different of them.
\begin{rem}
We could assume a weaker  upper bound on the Laplace transform for the diagonal entries such as the existence of $A$ finite such that
$$\int e^{tx }d\mu_{i,i}^N(x)\le \exp\{t^{2}+A|t|\},\quad \forall 1\le i\le N,$$
see the proof of Theorem \ref{rftheo}.
\end{rem}

\begin{Ex}\label{ex0}
\begin{enumerate}
\item Clearly a  centered Gaussian variable has a sharp sub-Gaussian Laplace transform. 
\item The Rademacher law  $B=\frac{1}{2} (\delta_{-1}+\delta_{1})$ satisfies a sharp sub-Gaussian Laplace transform since   for all real number $t$
$$T_B(t)=\cosh(t)\le e^{t^{2}/2}\,.$$
\item   $U$, the uniform law on the interval $[- \sqrt{{3}}, \sqrt{{3}} ]$, satisfies a sharp sub-Gaussian Laplace transform since  we have
$$\int x^{2 }dU(x)=1\,,$$
and
$$T_{U}(t)=\frac{1}{t\sqrt{{3}}} \sinh (t{\sqrt{{3}} })=\sum_{n\ge 0}\frac{t ^{2n} {3}^{n}}{(2n+1)!}\,.$$
Since for all $n\ge 0$, $ \frac{ {3}^{n}}{(2n+1)!}\le \frac{1}{2^{n} n!}$, it follows that $T_{U}(t)\le e^{\frac{t^{2}}{2}}$.

\item More generally  if  $\mu$ is a symmetric measure on $\R$    (i.e. such as $ \mu( - A) = \mu(A)$ for any Borel subset $A$ of $\R$) such that
$$\int x^{2} d\mu(x)=1,\quad \int x^{2n} d\mu(x)\le \frac{(2n)(2n-1)\cdots (n+1)}{2^{n}}\quad \forall n\ge 2$$
then $\mu$  satisfies a sharp sub-Gaussian Laplace transform.

\item If $X,Y$ are two independent variables with distribution $\mu$ and $\mu'$, two probability measures  which have a sharp sub-Gaussian Laplace transform, for any $a\in [0,1]$, the distribution of $\sqrt{a}X+\sqrt{1-a} Y$ has a sharp sub-Gaussian Laplace transform.
\item If $\mu_{i,j}^{N}=\mu$ for all $i,j$, then they satisfy a uniform lower bound on the Laplace transform. Also, if all the $\mu_{,j}^N$ are symmetric, the lower bound is automatically satisfied as the Laplace transform is lower bounded by $e^{\frac{1}{2}t^2}$.
\end{enumerate}
\end{Ex}
Note that many measures do not have a sharp sub-Gaussian Laplace transform, e.g. the sparse Gaussian law obtained by multiplying a Gaussian variable by a Bernoulli variable, or the well chosen sum of  Rademacher laws.
We will also  need that the empirical measure of the eigenvalues concentrates in a stronger scale than $N$, see Lemma \ref{convmun}. To this end we will also make the following classical assumptions to use standard concentration of measure tools. 
\begin{assum}\label{AC}
There exists a compact set $K$ such that the support of all $\mu_{i,j}^N$ is included in $K$ for all $i,j\in\{1,\ldots,N\}$and all integer number $N$, or all $\mu_{i,j}^N$ satisfy a log-Sobolev inequality with the same constant $c$ independent of $N$.
\end{assum}
\begin{rem} All the examples of  Example \ref{ex0} satisfy Assumption \ref{AC}, except possibly for sums of Gaussian variables and bounded entries.
\end{rem}

We then construct for all $N \in \N$, a real Wigner matrix $N \times N$ $X_{N}^{(1)}$ by setting :

$$
X_N^{(1)}(i,j) ={\bigg\lbrace}
\begin{array}{l}
\frac{ a_{i,j}^{(1)}}{\sqrt{N}} \text{ when } i \leq j, \cr
 \frac{ a_{j,i}^{(1)}}{\sqrt{N}} \text{ when } i > j\,.\cr
\end{array}$$
We denote $\lambda_{min}(X_{N}^{(1)})=\lambda_1\le \lambda_2 \cdots\le \lambda_N=\lambda_{\rm max}(X_{N}^{(1)})$ the eigenvalues of $X_N^{(1)}$.
It is well known \cite{Wig58} that under our  hypotheses the empirical distribution of the eigenvalues $\hat\mu_{X_{N}^{(1)}}^{N}=\frac{1}{N}\sum_{i=1}^N \delta_{\lambda_i}$ converges weakly towards the semi-circle distribution $\sigma$:
for all bounded continuous function $f$
$$\lim_{N\ra\infty}\int f(x) d\hat\mu_{X_N^{(1)}}^N(x)=\int f(x)d\sigma(x)=\frac{1}{2\pi}\int_{-2}^2 f(x) \sqrt{4-x^2} dx \qquad a.s. $$
It is also well known that the eigenvalues stick to the bulk since we assumed the entries have  sub-Gaussian moments \cite{FuKo,AGZ} :
$$\lim_{N\ra\infty}\lambda_{min}(X_{N}^{(1)})=-2\qquad \lim_{N\ra\infty}\lambda_{\rm max}(X_{N}^{(1)})=2,\qquad a.s$$
Our main result is a large deviation principle from this convergence. 
\begin{theo} \label{maintheow}  Suppose 
 Assumptions \ref{A0} and \ref{AC} hold. Then, 
the law of the largest eigenvalue  $\lambda_{\rm max}(X_N^{(1)})$ of $X^{(1)}_N$ satisfies a large deviation principle with speed $N$ and good rate function  $I^{(1)}$ which is infinite on $(-\infty,2)$ and otherwise given by 

\[ I^{(1)}( \rho)= \frac{1}{2}\int_2^{\rho} {\sqrt{x^2 - 4}}dx \,.\]
In other words, for any closed subset $F$ of $\mathbb R$,
$$\limsup_{N\rightarrow \infty }\frac{1}{N}\ln P\left(\lambda_{\rm max}(X_N^{(1)})\in F\right)\le -\inf_{F}I^{(1)}\,,$$
whereas for any open subset $O$ of $\mathbb R$
$$\liminf_{N\rightarrow \infty }\frac{1}{N}\ln P\left(\lambda_{\rm max}(X_N^{(1)})\in O\right)\ge -\inf_{O}I^{(1)}\,.$$
The same result holds for the opposite of the smallest eigenvalue $-\lambda_{min}(X_{N}^{(1)})$.

\end{theo}
Therefore, the large deviations principles  are the same as in the case of Gaussian entries as soon as the entries have a sharp sub-Gaussian Laplace transforms and are bounded, for instance for Rademacher variables or uniformly distributed variables. Hereafter we show how this result generalizes to other settings. 
First, this result extends to the case of Wigner matrices with complex entries as follows.  We now consider 
a family of independent random variables $(a^{(2)}_{i,j})_{1 \leq i\leq j \leq N}$, such that  the variables $a^{(2)}_{i,j}$ are distributed according to a law $\mu_{i,j}^N$ when $i \le j$, which are centered  probability measures on $\C$ (and on $\R$ if $i=j$).  We write $a^{(2)}_{i,j} = x_{i,j} + i y_{i,j}$ where $x_{i,j} = \Re(a^{(2)}_{i,j})$ and $y_{i,j} = \Im ( a^{(2)}_{i,j} )$. We suppose that for all $i \in [1, N]$, $y_{i,i} = 0$. In this context, for a probability measure on $\C$, we will consider its Laplace transform to be the function 
$$T_\mu (z):=  \int  \exp \{ \Re (a \bar{z})\} d\mu(a)\,.$$
 We  assume that 
\begin{assum}[A0c]\label{A0c} For all $i<j$
\[ \forall t \in \C , T_{\mu^{N}_{i,j}} (t) \leq \exp(|t|^2 / 4 ) \]
and  for all $i$
\[ \forall t \in \R, T_{\mu^{N}_{i,i}} (t) \leq \exp(t^2/2)\,. \]

We assume that for all $\delta>0$ there exists $\varepsilon(\delta)>0$ so that for all complex number $t$ with modulus bounded by $\varepsilon(\delta)$
$$T_{\mu_{i,j}^N} (t) \geq \exp\{\frac{|t|_2^2 }{4 }(1-\delta)\}, i<j,\qquad T_{\mu_{i,i}^N}(t)\ge  \exp \{\frac{(1-\delta) t^2  }{2} \}\,. $$
Moreover, for $\epsilon>0$ small enough $\sup_{|t|\le \epsilon}\sup_{i,j,N}| \partial_t^3\ln T_{\mu_{i,j}^N}(t)|$ is finite.
\end{assum}
Observe that the above hypothesis implies that for all $i<j$,  $2 \E[ x_{i,j}^2] = 2 \E[ y_{i,j}^2] =\E[ x_{i,i}^2]  = 1$ and $\E[ x_{i,j} y_{i,j}] = 0$. 
Examples of distributions satisfying Assumption \ref{A0c} are given by taking $(x_{i,j},y_{i,j})$ centered  independent variables with law satisfying  a sharp sub-Gaussian Laplace transform. Hereafter, we extend naturally Assumption \ref{AC} by assuming that the compact $K$ is a compact subset of $\mathbb C$ or log-Sobolev inequality holds in the complex setting.

We then construct for all $N \in \N$, $X^{(2)}_N$ a complex Wigner matrix $N \times N$ by letting :

$$
X_N^{(2)}(i,j) =\bigg\lbrace
\begin{array}{l}
\frac{ a^{(2)}_{i,j}}{\sqrt{N}} \text{ when } i \leq j \cr
\cr
 \frac{ \overline{ a^{(2)}_{j,i}}}{\sqrt{N}} \text{ when } i > j\cr
\end{array}$$
Again, it is well known that the spectral measure of $X_N^{(2)}$ converges towards the semi-circle distribution $\sigma$ and that the eigenvalues stick to the bulk \cite{AGZ}.  

\begin{theo} \label{maintheow2} Assume that Assumptions \ref{A0c} and \ref{AC} hold. Then, 
the law of the largest eigenvalue  $\lambda_{\rm max}(X_N^{(2)})$ of $X_N^{(2)}$ satisfies a large deviation principle with speed $N$ and good rate function  $I^{(2)}$ which is infinite on $(-\infty,2)$ and otherwise given by 
\[ I^{(2)}(\rho)=2I^{(1)}( \rho)=  \int_2^{\rho} \sqrt{x^2 - 4} dx \,.\]
\end{theo}
We finally generalize our result to the case of Wishart matrices. We let $L,M$ be two integers with $N=L+M$. Let $G_{L,M}^{(\beta)}$ be an $L\times M$ matrix with independent entries $(a_{i,j}^{(\beta)})_{1\le i\le L\atop 1\le j\le M}$ with  laws $\mu^{L,M}_{i,j}$   on the real line if $\beta=1$ and on the complex plane if $\beta=2$. The $\mu^{L,M}_{i,j}$ satisfy a sharp sub-Gaussian Laplace transform (with real  or complex values) for all $i,j\in [1,L]\times [1,M]$,  and its complementary uniform lower bound (Assumption \ref{A0}, or Assumption \ref{A0c}), are centered and  have covariance one. We set $W_{L,M}^{(\beta)}=\frac{1}{L}G_{L,M}^{(\beta)}(G_{L,M}^{(\beta)})^{*}$.  When $M/L$ converges towards $\alpha$, the spectral distribution of $W_{M,L}^{(\beta)}$ converges towards the Pastur-Marchenko law \cite{pastur-marchenko}: for any bounded continuous function $f$
$$\lim_{N\ra\infty} \int f(x)d\hat\mu^{L}_{W_{L,M}^{(\beta)}}(x)=\int f(x) d\pi_\alpha(x)\qquad a.s$$
where if $\alpha\ge 1$ and $a_\alpha= (1-\sqrt{\alpha})^2, b_\alpha= (1+\sqrt{\alpha})^2$,
$$\pi_{\alpha}(dx)=\frac{\sqrt{(b_\alpha-x)(x-a_{\alpha})}}{2\pi  x} \mathds{1}_{[a_\alpha,b_\alpha]} dx\,.$$
When $\alpha<1$, the limiting spectral measure has  aditionnally a Dirac mass at the origin with mass $1-\alpha$. We hereafter concentrates on the case $M\ge L$ up to replace $W_{L,M}^{(\beta)}$ by $(G_{L,M}^{(\beta)})^* G_{L,M}^{(\beta)}/ M$. Again, the extreme eigenvalues were shown to stick to the bulk \cite{BZ}. 
We prove a large deviation principle from this convergence:

\begin{theo} \label{maintheow} Assume that the $\mu_{i,j}^{N}$ satisfy Assumption \ref{AC}.
 Assume  they satisfy a sharp Gaussian Laplace transform \ref{A0} when $\beta=1$ or \ref{A0c} when $\beta=2$, and a uniform lower bounded Laplace transform \ref{A0} when $\beta=1$ or \ref{A0c} when $\beta=2$. Assume that there exists  $\alpha\ge 1$ and $\kappa>0$  so that $\frac{M}{L}-\alpha=o(N^{-\kappa})$. Then, 
the law of the largest eigenvalue  $\lambda_{\rm max}(W_{L,M}^{(\beta)})$ of $W_{L,M}^{(\beta)}$ satisfies a large deviation principle with speed $N$ and good rate function  $J^{({\beta})}$ which is infinite on $(-\infty,b_\alpha)$ and otherwise given by 

\[ J^{({\beta})}( x)= \frac{\beta}{4(1+\alpha)}\int_{b_{\alpha}}^{x} \frac{\sqrt{(y-b_\alpha)(y-a_{\alpha})}}{ y} dy \,.\]
where $\beta=1$ in the case of real entries, and $\beta=2$ in the case of complex entries.

\end{theo}
This problem can be seen as a generalization of the previous cases since if we consider the $N\times N$ matrix
\[X^{(w_\beta)}_N = 
\begin{pmatrix}
0 & \frac{1}{\sqrt{N}}G_{L,M}^{(\beta)} \\
\frac{1}{\sqrt{N} } (G^{(\beta)}_{L,M} )^{*}& 0 \\
\end{pmatrix}
\]
 the spectrum of the $N\times N$ matrix $X^{(w_{\beta)}}_N$ is given by $L$ eigenvalues $\sqrt{\frac{L}{N}\lambda}$, $L$ eigenvalues $-\sqrt{\frac{L}{N}\lambda}$, where $\lambda$ are the eigenvalues of $W_{L,M}^{(\beta)}$, and $M-L$ vanishing eigenvalues. Hence, the largest eigenvalue of $W_{L,M}^{(\beta)}$ is the square of the largest eigenvalue of $X_N^{(w_\beta)}$ multiplied by $N/L$. It is therefore equivalent to show a large deviation principle
 for the largest eigenvalue of 
 $X_N^{(w_{\beta})}$ with speed $N$ and  rate function
 $$I^{(w_{\beta})}(x)=J^{({\beta})}({({1+\alpha})}x^{2})\,.$$
 This  amounts to consider a Wigner matrix with some entries set to zero.  We denote $a_{i,j}^{(w_\beta)}$ the entries of $\sqrt{N} X_N^{(w_\beta)}$:
 \begin{eqnarray*}
 a_{i,j}^{(w_\beta)}&=& 0, \qquad \mbox{ if } i,j\le L \mbox{ or } i,j\ge L+1,\\
 a_{i,j}^{(w_\beta)}&=&a_{i-L,j}^{(\beta)},\quad  i\ge L+1, j\le L, \\
 a_{i,j}^{(w_\beta)}&=& \overline a_{j-L, i}^{(\beta)},\quad  j\ge L+1, i\le N.\end{eqnarray*}
Again, we denote by $\mu^{N}_{i,j}$ the law of the $i,j$th entry of this matrix. Hereafter, we denote by $\sigma_{w}$ the limiting spectral distribution of $X_{N}^{(w_\beta)}$ given for any test function $f$ by
$$\int f(x) d\sigma_{w}(x)=\frac{1}{1+\alpha}\left(\int f(\sqrt{\frac{x}{1+\alpha}})d\pi_{\alpha}(x)+\int f(-\sqrt{\frac{x}{1+\alpha}})d\pi_{\alpha}(x)\right)+\frac{\alpha-1}{\alpha+1} f(0)\,.$$
Therefore, we shall prove Theorem \ref{maintheow} by showing that
\begin{theo} \label{maintheob} Assume that the $\mu_{i,j}^{N}$ satisfy Assumption \ref{AC}.
 Assume  they satisfy a sharp Gaussian Laplace transform \ref{A0} when $\beta=1$ or \ref{A0c} when $\beta=2$, and a uniform lower bounded Laplace transform \ref{A0} when $\beta=1$ or \ref{A0c} when $\beta=2$. Assume that there exists  $\alpha\ge 1$ and $\kappa>0$  so that $\frac{M}{L}-\alpha=o(N^{-\kappa})$. Then, 
the law of the largest eigenvalue  $\lambda_{\rm max}(X_{N}^{(w_\beta)})$ of $X_{N}^{(w_\beta)}$ satisfies a large deviation principle with speed $N$ and good rate function  $I^{(w_{\beta})}$ which is infinite on $(-\infty, \tilde b_{\alpha})$,
if $\tilde b_{\alpha}= \sqrt{(1+\alpha)^{-1}  b_\alpha}$ and otherwise given by 

$$ I^{(w_{\beta})}( x)=\frac{ \beta}{1+\alpha}\int_{\tilde b_{\alpha}}^{x}\frac{ 1 }{y}\sqrt{ (1+\alpha)^{2}(y^{2}-1)^{2}-4\alpha}dy\,.$$
where $\beta=1$ in the case of real entries, and two in the case of complex entries.

\end{theo}

{\bf Acknowledgments:} Alice Guionnet wishes to thank A. Dembo for long discussions about large deviations for the largest eigenvalue for sub-Gaussian matrices in Abu Dhabi in 2011. The idea to  tilt measures by the spherical integral came out magically from a discussion with M. Potters in UCLA in 2017 and we wish to thank him for this beautiful inspiration. We also benifited from many discussions with M. Maida with whom one of the author is working on a companion paper on unitarily invariant ensembles, as well as with Fanny Augeri with whom we are working on a follow up paper tackling the general sub-Gaussian case. Finally, we are very grateful for stimulating discussions with O. Zeitouni and N. Cook. 

\medskip

This work was supported by the LABEX MILYON (ANR-10-LABX-0070) of Universit\'e de Lyon, within the program "Investissements d'Avenir" (ANR-11-IDEX- 0007) operated by the French National Research Agency (ANR).

\subsection{Scheme of the proof}
The idea of the proof is reminiscent of Cramer's approach to large deviations: we appropriately  tilt measures to make the desired deviations likely. The point is to realize that it is enough to shift the measure in a random direction and use estimates on spherical intergrals obtained by one of the author  and M. Maida \cite{GuMa05}. To be more precise,
we shall follow the usual scheme to prove first exponential tightness:
\begin{lemma}\label{exptight} For $\beta=1,2,{w_{1}}, {w_{2}}$, assume that the distribution of the entries $a^{(\beta)}_{i,j}$ satisfy Assumption \ref{AC} for $\beta=1,{w_{1}}$ and Assumption \ref{A0c} for $\beta=2,w_{2}$. Then:

\[ \lim_{K \to + \infty} \limsup_{N \to \infty} \frac{1}{N} \ln \Pp[ \lambda_{\max}(X_N^{(\beta)}) > K] = - \infty \]
Similar results hold for $\lambda_{\min}(X_N^{(\beta)})$.
\end{lemma}
This result is proved in Section \ref{exptightsec}.
Therefore it is enough to prove a weak large deviation principle. 

In the following we summarize the  assumptions on the distribution of the entries as follows :
\begin{assum} \label{ass} 
Either the $\mu_{i,j}^{N}$ are uniformly compactly supported in the sense that there exists a compact set $K$ such that the support of all $\mu_{i,j}^N$ is included in $K$, or the  $\mu_{i,j}^N$ satisfy a uniform log-Sobolev inequality in the sense that there exists a constant $c$ independent of $N$ such that for all smooth function $f$
$$\int f^{2}\ln\frac{f^{2}}{\mu_{i,j}^{N}(f^{2})}d\mu^{N}_{i,j}\le c\mu^{N}_{i,j}(\|\nabla f\|_2^{2})\,.$$
When $\beta=1,w_{1}$ $\mu^{N}_{i,j}$ satisfy Assumption \ref{A0}, when $\beta=2,w_{2}$, they satisfy Assumption \ref{A0c}. In the case of  Wishart matrices, $\beta= w_{1}$ or $w_{2}$, we assume that there exists $\alpha >1$ and $\kappa>0$ so that $|\frac{M}{L}-\alpha|\le N^{-\kappa}$ for $N$ large enough. % Moreover, for $\epsilon>0$ small enough $\sup_{|t|\le \epsilon}\sup_{i,j,N}| \partial_t^3\ln T_{\mu_{i,j}^N}(t)|$ is finite.
\end{assum}

We shall first prove that we have a weak large deviation upper bound:
\begin{theo}\label{theowldub} Assume that Assumption \ref{ass} holds. Let $\beta=1,2,w_{1},w_{2}$. Then,  for any real number $x$,
$$\limsup_{\delta\ra 0}\limsup_{N\ra\infty} \frac{1}{N}\ln \Pp\left(\left|\lambda_{\rm max }(X_N^{(\beta)})-x\right|\le\delta\right)\le -I_\beta(x)$$
\end{theo}
We shall then obtain  the large deviation lower bound.
\begin{theo} \label{theowllb}Assume that Assumption \ref{ass} holds.  Let $\beta=1,2,w_{1},w_{2}$. 
Then, for any real number $x$,
$$\liminf_{\delta\ra 0}\liminf_{N\ra\infty}\frac{1}{N}\ln \Pp\left(\left|\lambda_{\rm max}(X_N^{(\beta)})-x\right|<\delta\right)\ge -I_\beta(x)$$
\end{theo}

To prove Theorem \ref{theowldub}, we first 
 show that the rate function is infinite below the right edge of   the support of the limiting spectral distribution. To this end, we use that the spectral measure $\mun$ converges towards its limit which much larger probability. We denote this limit $\sigma_{\beta}$: $\sigma_{1}=\sigma_{2}=\sigma$ and $\sigma_{w_{1}}=\sigma_{w_{2}}=\sigma_{w}$. We let $d$ denote the Dudley distance:
 $$d(\mu,\nu)=\sup_{\|f\|_L\le 1}\left|\int f(x) d\mu (x) -\int f(x) d\nu(x)\right|\,,$$
 where $\|f\|_L=\sup_{x\neq y}\left|\frac{f(x)-f(y)}{x-y}\right| +\sup_x|f(x)|\,.$
 
\begin{lemma}\label{convmun}  Assume  that the 
$\mu_{i,j}^N$ are uniformly compactly supported or satisfy a uniform log-Sobolev inequality, as well as, in the case $w_{1},w_{2}$, that there exists $\kappa>0$ such that $|\frac{M}{N}-\alpha|\le N^{-\kappa}$. 
Then,  for $\beta=1,2,w_{1},w_{2}$, there exists $\kappa'\in (0,\frac{1}{10}\wedge \kappa)$ such that

\[ \limsup_{N \to \infty} \frac{1}{N} \ln \Pp\left( d( \hat \mu^{N}_{X_N^{(\beta)}}, \sigma_\beta) > N^{-\kappa'} \right) = - \infty \,.\]
\end{lemma}
The proof of this lemma is given in the appendix. 
As a consequence, we deduce that the extreme eigenvalues can not deviate towards a point inside the support of the limiting spectral measure with probability greater than $e^{-N^{1+\kappa}}$
and therefore
\begin{cor} Under the assumption of Lemma \ref{convmun},  For $\beta=1,2$ let $x$ be a real number in $(-\infty,2)$ or, for $\beta={w_{1}},w_{2}$,  take $x\in (-\infty, \tilde b_{\alpha})$. 
Then, for $\delta>0$ small enough,
$$\limsup_{N\ra\infty}\frac{1}{N}\ln \Pp\left(| \lambda_{\rm max}(X_{N}^{(\beta)})-x|\le \delta\right)=-\infty\,.$$
\end{cor}
Indeed,  as soon $\delta>0$ is small enough so that   $x+\delta$ is smaller than $2-\delta$ for $\beta=1,2$ (resp $b_\alpha-\delta$ for $\beta=w_{1},w_{2}$),   $d(\mun, \sigma_{\beta})$ is bounded below by some $\kappa(\delta)>0$ on $|\lambda_{\max}(X_N^{(\beta)})-x|\le \delta$. Hence, Lemma \ref{convmun} implies the Corollary.

In order to prove the weak large deviation bounds for the remaining $x$'s, we shall tilt the measure by using spherical integrals:
$$I_N(X, \theta)=\E_e[ e^{\theta N\langle e, X e\rangle}]$$
where the expectation holds over $e$ which follows the uniform measure on the sphere $\mathbb S^{N-1}$ with radius one. 
The asymptotics of 
$$J_N(X,\theta)=\frac{1}{N}\ln I_N(X,\theta)$$ 
were studied in \cite{GuMa05} where it was proved that
\begin{theo}\cite[Theorem 6]{GuMa05}\label{myl}

If $(E_N)_{N \in \N}$ is a sequence of $N \times N$ real symmetric matrices when $\beta=1$ and complex Gaussian matrices when $\beta=2$ such that :
\begin{itemize}
\item The sequence of empirical measures $\hat\mu^{N}_{E_N}$ weakly converges to a compactly supported measure $\mu$,
\item There are two reals $\lambda_{\min}(E), \lambda_{\max}(E)$ such that $\lim_{N \to \infty} \lambda_{\min}(E_N) = \lambda_{\min}(E)$ and $\lim_{N \to \infty} \lambda_{\max}(E_N) = \lambda_{\max}(E)$,
\end{itemize}
and $\theta \geq 0$, then : 
\[ \lim_{N \to \infty} J_N(E_N,\theta) = J(\mu,\theta, \lambda_{\max}(E)) \]
\end{theo}
The limit $J$ is defined as follows.  For a compactly supported probability measure we define its Stieltjes transform $G_\mu$ by
\[ G_{\mu}(z) := \int_{\R} \frac{1}{z-t} d \mu(t) \]

We assume hereafter that $\mu$ is supported on a compact $[a,b]$. Then  $G_{\mu}$ is a bijection from $\R \setminus [a,b]$ to $] G_{\mu}(a), G_{\mu}(b) [ \setminus \{ 0 \}$ where $G_{\mu}(a), G_{\mu}(b)$ are taken as the limits of $G_{\mu}(t)$ when $t \to a ^{-}$ and $t \to b ^{+}$. We denote by $K_{\mu}$ its inverse and let $R_{\mu}(z) := K_{\mu}(z) - 1/z$ be its $R$-transform as defined by Voiculescu in \cite{Vo5} (defined on $] G_{\mu}(a), G_{\mu}(b) [ $). In the sequel, for any compactly supported probability measure $\mu,$ we denote by $r(\mu)$ the right edge of the support of $\mu.$
In order to define the rate function, we now introduce, for any $\theta \ge 0,$   and $\lambda \ge r(\mu)$,
$$
J( \mu, \theta, \la):= \theta v(\theta, \mu, \la) -\frac{\beta}{2} \int \log\left(1+\frac{2}{\beta} \theta v(\theta, \mu, \la) - \frac{2}{\beta}\theta y\right) d\mu( y),
$$
with 
$$
 v(\theta, \mu, \la) := \left\{\begin{array}{ll}
                                 R_\mu(\frac{2}{\beta} \theta), & \textrm{if } 0 \le \frac{2 \theta}{\beta} \le H_{\rm max}(\mu, \lambda) := \lim_{z \downarrow \lambda} \int \frac{1}{z-y} d\mu( y),\\
                                 \lambda - \frac{\beta }{2\theta}, & \textrm{if }\frac{2 \theta }{\beta}>  H_{\rm max}(\mu, \la),
                                \end{array}
\right.
$$
We shall later use that spherical integrals are continuous.  We recall here Proposition 2.1 from \cite{Ma07} and  Theorem 6.1  from \cite{GuMa05}. We denote by $\|A\|$ the operator norm of the matrix $A$ given by
$\|A\|=\sup_{\|u\|_2=1}\| Au\|_2$ where $\|u\|_2=\sqrt{\sum |u_i|^2}$.  
\begin{prop} \label{1}
For every $\theta >0$, every $\kappa \in ]0, 1/2[$, every $M>0$, there exist a function $g_{\kappa} : \R^+ \to \R^+$ going to 0 at 0 such that
for any $\delta > 0$ and $N$ large enough, with $B_N$ and $B'_N$ such that $d(\hat \mu^{N}_{B_N},\hat \mu^{N}_{B'_N}) < N^{- \kappa}$, $ |\lambda_{\rm max}(B_N) - \lambda_{\rm max}(B_N) | < \delta$ and $\sup_{N} ||B_N|| \le M $, $\sup_{N} ||B'_N|| \le M$ :
\[ |J_N(B_N, \theta) - J_N(B'_N, \theta)| < g_{\kappa}(\delta) \,.\] 
\end{prop}
From Theorem \ref{myl} and Proposition \ref{1}, we  deduce that : 

\begin{cor}\label{cont}
For every $\theta >0$,  every $\kappa \in ]0, 1/2[$,  every $M>0$, for any $\delta > 0$ and  $\mu$ a probability measure supported in $[-M,M]$,  if we denote by  $\mathcal{B}_N$ the set of symmetric matrices $B_N$ such that $d( \mu_{B_N}, \mu) < N^{- \kappa}$, $ |\lambda_{\rm max}(B_N) - \rho | < \delta$, and $\sup_{N} ||B_N|| \le M $,   for $N$ large enough, we have  : 

\[ \limsup_{N \to \infty} \sup_{B_N \in \mathcal{B}_N} |J_N(B_N, \theta) - J(\mu,\theta, \rho) | \leq 2 g_{\kappa}(\delta) \]
where $g_{\kappa}$ is the function in Proposition \ref{1}.
\end{cor}
By Lemma \ref{exptight} and Lemma \ref{convmun}, it is enough to study the probability of deviations on the set where $J_N$ is continuous:
\begin{cor} \label{app} Suppose Assumption \ref{AC} holds. For $\delta>0$, take a real number $x$
and set  for $M$ large (larger than $x+\delta$ in particular), $\mathcal A_{x,\delta}^{M}$ to be the set of $N\times N$ self-adjoint matrices given by
$$\mathcal A_{x,\delta}^{M}=\lbrace X: \left|\lambda_{\rm max}(X)-x\right|<\delta\rbrace\cap \{X: d(\hat\mu_X^{N}, \sigma_\beta)< N^{-\kappa'}\} \cap \{ X:\|X\|\le M\}\,, $$
where $\kappa'$ is  chosen as in Lemma \ref{convmun} . 
Let $x$ be a real number, $\delta>0$ and $\kappa'$ as in Lemma \ref{convmun}. Then, for any $L>0$,  for $M$ large enough
$$ \Pp\left(\left|\lambda_{\rm max}(X_N^{(\beta)})-x\right|<\delta \right)=\Pp\left(X_N^{(\beta)}\in \mathcal A_{x,\delta}^{M}\right)+ O(e^{-N L})\,.$$
\end{cor} 
We are now in position to get an upper bound for $\Pp\left( X_N^{(\beta)}\in \mathcal A_{x,\delta}^{M}\right)$. 
In fact, by the continuity of spherical integrals of Corollary \ref{cont},  for any $\theta\ge 0$, 
\begin{eqnarray}
\Pp\left( X_N^{(\beta)}\in \mathcal A_{x,\delta}^{M}\right)&=&\E\left[ \frac{I_N(X_N^{(\beta)},\theta)}{I_{N}(X_{N}^{(\beta)} ,\theta)}1_{
 \mathcal A_{x,\delta}^{M}
}\right]\nonumber\\
&\le &\E[ I_N(X_N^{(\beta)},\theta)] \exp\{ - N\inf_{X\in \mathcal A_{x,\delta}^{M}} J_N(X,\theta)\}\nonumber\\
&\le&  \E[ I_N(X_N^{(\beta)},\theta)]\exp \{N(2g_\kappa(\delta)-  J(\sigma_\beta, \theta, x))\} \label{jh}
\end{eqnarray}
where we used that $x\rightarrow  J(\sigma_\beta, \theta, x)$ is continuous and took $N$ large enough.
It is therefore central to derive the asymptotics of $$F_N(\theta,\beta)=\frac{1}{N}\ln  \E[ I_N(X_N^{(\beta)},\theta)]$$ and we shall
prove in section \ref{rf} that
\begin{theo}\label{rftheo} Suppose Assumption \ref{ass} holds. 
For $\beta=1,2,{w_{1}},w_{2}$ and $\theta\ge 0$,
$$ \lim_{N\ra\infty}F_N(\theta,\beta) = F(\theta,\beta)$$
with $F(\theta,\beta)=\theta^2/\beta$ if $\beta=1,2$ and when $\beta={w_{i}}$, $i=1,2$:
$$F(\theta,w_i)=\sup_{x\in [0,1]}\{ \frac{2\theta^{2}}{i} x(1-x)+\frac{i}{2(1+\alpha)}\ln(1-x)+\frac{i \alpha }{2(1+\alpha)}\ln x\}
-i C_{\alpha}\,,$$
where $C_{\alpha}=\frac{1}{2 (1+\alpha)}
\ln (\frac{1}{1+\alpha})+\frac{\alpha}{2(1+\alpha)} \ln \frac{\alpha}{1+\alpha}$
\end{theo}
We therefore deduce from \eqref{jh}, Corollaries \ref{app} and \ref{cont} , and Theorem \ref{rftheo}, by first letting $N$ going to infinity, then $\delta$ to zero and finally $M$ to infinity, that 
$$
\limsup_{\delta\ra 0}\limsup_{N\ra\infty}\frac{1}{N}\ln  \Pp\left(\left|\lambda_{\rm max}(X_N^{(\beta)})-x\right|<\delta \right)\le   F(\theta,\beta)-  J(\sigma_\beta,\theta, x)\,.$$
We next  optimize over $\theta$ to derive the upper bound:
\begin{equation}
\label{b1}
\limsup_{\delta\ra 0}\limsup_{N\ra\infty}\frac{1}{N}\ln  \Pp\left(\left|\lambda_{\rm max}(X_N^{(\beta)})-x\right|<\delta \right)\le   -\sup_{\theta\ge 0} \{ J(\sigma_\beta,\theta,x)- F(\theta,\beta)\}\,.
\end{equation}
To complete the proof of Theorem \ref{theowldub}, we show in section \ref{grf} that, with the notations of Theorems \ref{maintheow},\ref{maintheow2}, and \ref{maintheob}, 
\begin{prop}\label{propgrf}
For $\beta=1,2,{w_{1}},w_{2}$,
$$I_\beta(x)= \sup_{\theta\ge 0} \{ J(\sigma_\beta,\theta,x)- F(\theta,\beta)\}\,.$$
\end{prop}
To prove the complementary lower bound, we shall prove that
\begin{lemma} \label{cruc}
For $\beta=1,2$, 
for any $x>2$ and for $\beta={w_{1}}, w_{2}$ for any $x>\tilde b_{\alpha}$, there exists $\theta=\theta_x\ge 0$ such that for any $\delta>0$
and $M$ large enough,

$$\liminf_{N\ra\infty} \frac{1}{N} \ln \frac{   \E[ \mathds{1}_{ X_N^{(\beta)}\in \mathcal A_{x,\delta}^{M}}
I_N(X_N^{(\beta)},\theta)]}{ \E[ I_N(X_N^{(\beta)},\theta)]}\ge 0\,.$$
\end{lemma}
This lemma is proved by showing that the matrix whose law has been tilted by the spherical integral is approximately a rank one perturbation of a Wigner matrix, from which we can use the techniques developped to study the famous BBP transition \cite{BBP}. The conclusion follows since then
\begin{eqnarray*}
\Pp\left( X_N^{(\beta)}\in \mathcal A_{x,\delta}^{M}\right)&\ge &\frac{   \E[ \mathds{1}_{ X_N^\delta\in \mathcal A_{x,\delta}^{M}}
I_N(X_N^{(\beta)},\theta_{x})]}{ \E[ I_N(X_N^{(\beta)},\theta_x)]} \E[ I_N(X_N^{(\beta)},\theta_x)] \exp\{ - N\sup_{X\in \mathcal A_{x,\delta}^{M}} J_N(X,\theta_{x})\}\nonumber\\
&\ge& \exp\{N(g_\kappa(\delta)+F(\theta_{x},\beta)- J(\sigma_\beta, \theta_{x}, x)+o(\delta))\} \\
&\ge& \exp\{-N I_{\beta }(x)-No(\delta)\} \\
\end{eqnarray*}
where we finally used Theorem \ref{rftheo} and Lemma \ref{cruc}.

\section{Exponential tightness}\label{exptightsec}
In this section we prove Lemma \ref{exptight}. We will use a standard net argument that we recall for completness.
For $N \in \N$, let $R_N$ be a $1/2$-net of the sphere (i.e. a subset of the sphere $\mathbb{S}_{N-1}$ such as for all $u \in \mathbb{S}_{N-1}$ there is $v \in R_N$ such that $||u - v||_2 \leq 1/2$. Here the sphere is inside $\R^{N}$ for $\beta=1,w_{1}$ and $\C^{N}$ for $\beta=2,w_{2}$). We know that we can take $R_N$ with cardinality smaller than  ${3}^{\beta N}$. We  notice that for $M >0$

\begin{equation}\label{b2}\Pp[ ||X_N^{(\beta)}|| \ge 4 K] \leq  9^{\beta N} \sup_{u,v \in R_{N}}\Pp[\langle X_N^{(\beta)} u, v \rangle \ge  K  ] \end{equation}
Indeed, if we denote, for $v\in\mathbb S^{N-1}$, $u_{v}$ to be an element of $R_{N}$ such that $\|u_{v}-v\|_{2}\le 1/2$,
$$\|X_{N}^{(\beta)}\|=\sup_{v\in\mathbb S^{N-1}}\|X_{N}^{(\beta)}v\|_{2}\le 
\sup_{v\in\mathbb S^{N-1}}(\|X_{N}^{(\beta)}u_{v}\|_{2} +\frac{1}{2}\|X_{N}^{(\beta)}\|)$$
so that 
\begin{equation}\label{b3} \|X_{N}^{(\beta)}\|\le 2 \sup_{u\in R_{N}}\|X_{N}^{(\beta)}u\|_{2}\end{equation}
Similarly, taking $v= \frac{X_{N}^{(\beta)}u}{\|X_{N}^{(\beta)}u\|_{2}}$, we find 
$$\|X_{N}^{(\beta)}u\|_{2}=\langle v,X_{N}^{(\beta)}u\rangle\le  \langle u_{v},X_{N}^{(\beta)}u\rangle+\|v-u_{v}\|_{2}\|X_{N}^{(\beta)}v\|_{2}$$
from which  we deduce that 
$$ \|X^{N}_{\beta}\|\le 4 \sup_{u,v\in R_{N}} \langle X_N^{(\beta)} u, v \rangle $$
and  \eqref{b2} 
 follows.
We next bound the probability of deviations of $ \langle X_N^{(\beta)} v, u \rangle$ by using Tchebychev's inequality.
For $\theta\ge 0$ we indeed have

\begin{eqnarray}
\Pp[\langle X_N^{(\beta)} u, v \rangle \ge   K  ] &\leq& \exp\{-\theta N K\} \E[ \exp\{N\theta \langle X_N^{(\beta)} u, v \rangle\} ] \nonumber\\
& \leq & \exp\{-\theta N K\}\E[ \exp \left\{
\sqrt{N}\left(2 \sum_{i<j} \Re (a_{i,j}^{(\beta)}u_i \bar v_j ) + \sum_i a_{i,i} u_i v_i \right)\right\}]\nonumber\\
& \leq & \exp\{-\theta N  K\} \exp \left(\frac{\theta^2 N}{\beta'} (2\sum_{i<j} |u_{i}|^{2} |v_{j}|^{2}+\sum_i  |u_i |^{2}|v_i|^2) \right) \label{qw}
\end{eqnarray}
where we used that the entries have a sharp sub-Gaussian Laplace transform. In the case of Wishart matrices, we bounded above some vanishing contributions by a non-negative term. When $\beta=w_{i}$, $\beta'=i$, otherwise
$\beta'=\beta$. We can now complete the upper bound:

\begin{eqnarray*}
\Pp[\langle X_N^{(\beta)} u, v \rangle \ge  K  ]
%& \leq & \exp\{-\theta M\} \exp \left(\frac{\theta^2}{N} \sum_{i<j} \frac{(u_i v_j + u_j v_i)^2}{2} + \sum_i  (u_i v_i)^2 \right)\\
& \leq & \exp \left( \frac{\theta^2 N}{\beta' } \frac{ ||u||_2^2 ||v||_2^2 + \langle u, v \rangle^2}{2} - \theta N K \right) \\
&\leq & \exp \left( N \left( \frac{1}{\beta' } - K \right) \right) 
\end{eqnarray*}
where we took 
 $\theta = 1$. We conclude that : 

\[ \Pp[\langle X_N^{(\beta)} u, v \rangle \ge   K  ] \leq \exp \left(  N(1-K) \right) \]
This complete the proof of the Lemma with \eqref{b2}.

\section{Proof of Theorem \ref{rftheo} }\label{rf}

We consider in this section a random unitary vector $e$ taken uniformly on the sphere $\mathbb S^{N-1}$ and independent of $X_{N}^{(\beta)}$. 
We define $F_N$ by setting,  for $\theta > 0$ :

\[ F_N( \theta,\beta) = \frac{1}{N} \ln \E_{X_{N}^{(\beta)}} \E_{e}[ \exp(N \theta \langle e, X_{N}^{(\beta)} e \rangle) ] \]
where we take both the expectation $\E_{e}$ over $e$ and the expectation $\E_{X_{N}^{(\beta)}}$ over  $X_{N}^{(\beta)}$.
In this section we derive the asymptotics of $F_{N}(\theta,\beta)$. $F(\theta,\beta)$ is as in Theorem \ref{rftheo}.
 We prove a refinment of Theorem \ref{rftheo},  which shows that under our assumption
 of sharp sub-Gaussian tails, the random vector $e$ stays delocalized under the tilted measure.
 \begin{prop}\label{tyui}
Suppose Assumption \ref{A0} holds if $\beta=1,w_1$ and Assumption \ref{A0c} holds if $\beta=2,w_2$. Denote by  $V_{N}^{\epsilon}=\{ e \in \mathbb{S}^{N-1} \, :\, \forall i, | e_i | \leq N^{-1/4 - \epsilon} \}$. Then, for 
$\epsilon\in (0,\frac{1}{4})$,
\begin{eqnarray*}
F(\theta,\beta)&=&\lim_{N \to + \infty}  F_N( \theta,\beta) =\lim_{N\ra\infty}\frac{1}{N} \ln  \E_e[\mathds{1}_{e\in V_N^\epsilon}\E_{X_{N}^{(\beta)}} [\exp(N \theta \langle e, X_{N}^{(\beta)} e \rangle)] ] \\
\end{eqnarray*}
\end{prop}
We first consider the case of Wigner matrices and then the case of Wishart matrices: in both cases the proof shows that the above delocalization holds  (i.e we can restrict ourselves to vectors $e$ in  $V_{N}^{\epsilon}$) and we shall not mention it in the following statements. 

\subsection{Wigner matrices}
In this section we prove Theorem \ref{rftheo} in the case of Wigner matrices, namely:

\begin{lemma}\label{rfwig} Suppose Assumption \ref{A0} holds if $\beta=1$ and Assumption \ref{A0c} holds if $\beta=2$. Then for
any $\theta\ge 0$

\[ \lim_{N \to + \infty}  F_N( \theta,\beta) =F(\theta,\beta)=\frac{ \theta^2}{\beta}\,. \]

\end{lemma}

\begin{proof}
By denoting $L_{\mu} = \ln T_{\mu}$,  we have : 
\begin{eqnarray*}\E_{X_{N}^{(\beta)}}[ \exp(N \theta \langle e, X_N^{(\beta)} e \rangle) ] 
&= & 
 \E_{X_{N}^{(\beta)}}[ \exp\{ \sqrt{N} \theta (2 \sum_{i<j}\Re(a_{i,j}^{(\beta)}e_{j}\bar e_{i})+\sum_{i}a_{i,i}^{(\beta)}|e_{i}|^{2})\} ] \\
&=&
\exp\{ \sum_{i<j} L_{\mu_{i,j}^{N}}(2 \theta \bar e_i e_j \sqrt{N})  +\sum_i L_{\mu_{i,i}^{N}}(\theta |e_i|^2 \sqrt{N} )\}  \\
\end{eqnarray*}
where we used the independence of the $(a_{i,j}^{(\beta)})_{i\le j}$.
Using that the entries have a sharp sub-Gaussian Laplace transform (using on the diagonal the weaker bound
 $L_{\mu_{i,i}^{N}}(t) \leq \frac{1}{\beta} t^2 +A|t|$)  and $\sum e_i^2 = 1$, we deduce that: 

\begin{eqnarray*}
\E_{X_{N}^{(\beta)}}[ \exp(N \theta \langle e, X_N^{(\beta)} e \rangle) ] & \leq& \E_{e}  [  \exp\{\frac{2 N\theta^{2}}{\beta} \sum_{i<j} |e_{i}|^{2}|e_{j}|^{2}+
\frac{N\theta^{2}}{\beta} \sum_{i} |e_{i}|^{4}+
A\sqrt{N}\theta \sum_{i} e_{i}^{2}\}]\\
&\le &
\exp( N \frac{\theta^2}{\beta} +A\sqrt{N} \theta ) \end{eqnarray*}

So that  we have proved the upper bound that 
\begin{equation}\label{ub}
 \limsup_{N\ra\infty} F_N( \theta,\beta) \leq  \limsup_{N\ra\infty}  \sup_{e\in \mathbb S^{N-1}} \frac{1}{N}\ln \E_{X_{N}^{(\beta)}}[ \exp(N \theta \langle e, X_N^{(\beta)} e \rangle) ]\le
 \frac{ \theta^2 }{\beta} \end{equation}
 We next prove the corresponding lower bound. The idea is that the expectation over the vector $e$ concentrates on delocalized eigenvectors with entries so that $\sqrt{N} e_{i }\bar e_{j}$ is going to zero for all $i,j$. As a consequence we will be able to use the uniform lower bound on the Laplace transform to lower bound $F_{N}(\theta,\beta)$.

Let $V_{N}^{\epsilon}=\{ e \in \mathbb{S}^{N-1} \, :\, \forall i, | e_i | \leq N^{-1/4 - \epsilon} \}$  be the subset of the sphere $\mathbb{S}^{N-1}$ with entries smaller than $N^{-1/4-\epsilon}$ for some $\epsilon\in (0,\frac{1}{4})$. We have that :  

$$
\E [ \exp ( N \theta \langle e, X^N_\beta e \rangle ) ]  \geq  \E_{e}  [ \mathds{1}_{e\in V^{\epsilon}_N} \prod_{i < j} \exp\{L_{\mu^{N_{i,j}}}( 2 \sqrt{N} \theta \bar e_i e_j )\} \prod_{i} \exp\{  L_{\mu^{N}_{{i,i}}}(\sqrt{N} \theta |e_i|^2 )\}] 
$$
 For $e \in V^{\epsilon}_N$, $2 \sqrt{N} \theta |e_i e_j| \leq 2\theta N^{- \epsilon}$ so that :
\[ \lim_{N \to + \infty} \sup_{e \in V^{\epsilon}_N} |2 \sqrt{N} \theta e_i e_j| = 0 \]
By the uniform lower bound on the Laplace transform of Assumptions \ref{A0} or \ref{A0c}, we deduce that for any $\delta>0$

\begin{equation}\label{mn0}
\E [ \exp ( N \theta \langle e, X^N_\beta  e \rangle ) ]  \geq \Pp_e[V^{\epsilon}_N] e^{N \frac{\theta^2}{\beta}  (1- \delta)} \,.
\end{equation}
We shall use that 
\begin{lemma} \label{vec}For any  $\epsilon\in (0,1/4)$ we have
\[ \lim_{N \to \infty} \Pp_{e}[ e \in V^{\epsilon}_N ] =1 \]\,.
\end{lemma}
As a consequence, we deduce from \eqref{mn0} that for any $\delta>0$ and $N$ large enough
\[ \liminf_{N\to \infty} F_N( \theta,\beta) \geq (1 - \delta)\frac{ \theta^2}{\beta} \]

So that together with \eqref{ub}  we have proved  the announced limit

\[ \lim_{N\to \infty} F_N ( \theta,\beta) = \frac{ \theta^2 }{\beta}\]
which completes the proof of Lemma \ref{rfwig}.

Finally we prove Lemma \ref{vec}. To this end we use the well known representation of  the vector $e$ as a renormalized (real or complex) Gaussian vector:
$$e=\frac{g}{\|g\|_{2}}$$
where $g = (g_1,...,g_N)$ is a Gaussian vector of covariance matrix $I_N$. By the law of large numbers, we have the following almost sure limit : 

\[ \lim_{N \to \infty} \frac{||g||_{2}}{\sqrt{N}} = 1 \]
We also have by the union bound 

\[ \Pp[ \exists i \in [1,N], |g_i| >  N^{1/4 - \epsilon} /2 ] \leq N \Pp[ |g_1| >  N^{1/4 - \epsilon}/2]\le N\exp\{-\frac{1}{4} N^{1/2-2\epsilon}\} \]
from which the result follows. 

\end{proof}

\subsection{Wishart matrices}
In this subsection we prove Theorem \ref{rftheo} in the case of Wishart matrices, namely:
\begin{lemma}\label{b66}
Let $\beta=w_{1}$ or $w_{2}$.  Suppose Assumption \ref{ass} holds. Then for any $\theta\ge 0$, for $i=1,2$
$$\lim_{N\ra\infty}F_{N}(\theta, w_{i})=F(\theta,w_i)=\sup_{x\in [0,1]}\{ \frac{2\theta^{2}}{i} x(1-x)+\frac{i}{2(1+\alpha)}\ln(x)+\frac{i\alpha}{2(1+\alpha)}\ln (1-x)\}-i C_{\alpha}\,,$$
where $C_{\alpha}=\frac{1}{2 (1+\alpha)}
\ln (\frac{1}{1+\alpha})+\frac{\alpha}{2(1+\alpha)} \ln \frac{\alpha}{1+\alpha}$. Moreover, the supremum is achieved at a unique $x_{\theta,\alpha}$ in $[0,1]$ (as it maximizes a strictly concave function). $x_{\theta,\alpha}$ is the almost sure limit of $\|e_{1}\|_{2}^{2}$, the norm of the first $L$ entries of $e$,   under the tilted law
$$d\Pp^{\theta}(e)=\frac{\E_{X}[\exp\{\theta N\langle e, X^{N}_{\beta} e\rangle\} ]d\Pp(e)}{\E_{e}[\E_{X}[\exp\{\theta N\langle e, X^{N}_{\beta} e\rangle\} ]]}\,.$$

\end{lemma}

\begin{proof}
We have, with the same notations than in the previous case :

\begin{eqnarray*}
 \E_{X_{N}^{w_{i}}} [ \exp(  N\theta  \langle X_N^{w_{i}} e,e \rangle ) ] 
&=& \exp \left\{ \sum_{ 1 \leq i \leq M, 1 \leq j \leq L  } L_{\mu^{N}_{i,j}}( \sqrt{N} 2 \theta e^{(1)}_i \bar e^{(2)}_j) \right\}\\
\end{eqnarray*}
where $e=(e^{(1)},e^{(2)})$, that is $e^{(1)}$ is the vector made of  the $L$ first entries of $e$ and $e^{(2)}$ the vector made of  the $M$ last entries of $e$.
Using that 
the $\mu^{N}_{i,j}$ have a sharp sub-Gaussian Laplace transform and a uniform lower bounded Laplace transform, we deduce that with $V_{N}^{\epsilon}=\{ e\in\mathbb S^{N-1}\,:\, |e_{i}|\le N^{-1/4-\epsilon}\}$ we find that for any $\delta>0$ and
$N$ large enough
\begin{equation}\label{b5}
\E_{e}[\mathds{1}_{V_{N}^{\epsilon}} \exp\{(1-\delta) \frac{2\theta^{2}}{ i}N\|e^{(1)}\|_{2}^{2}\|e^{(2)}\|^{2}_{2}\}]
\le \E_{X_{N}^{w_{i}}}[I_N( \theta,w_{i}) ]\le \E_{e}[ \exp\{ \frac{2\theta^{2}}{ i}N\|e^{(1)}\|_{2}^{2}\|e^{(2)}\|^{2}_{2}\}]
\end{equation}
where  $\|e^{(1)}\|_{2}^{2}=1-\|e^{(2)}\|_{2}^{2}$  follows a Beta law with parameters $(i L/2,i M/2)$, so its distribution is given by
$$ {\rm Beta}_{iM/2,iL/2}(dx)=C_{M,L} x^{i L/2}(1-x)^{i M/2}\mathds{1}_{x\in [0,1]} dx\,,$$
with $C_{M,L}=\Gamma(iN/2)/\Gamma(iM/2)\Gamma(iL/2)$.
Therefore, Laplace method implies that
\begin{equation}\label{b6}
\lim_{N\ra\infty}\frac{1}{N}\ln  \E_{e}[ \exp\{ \frac{2\theta^{2}}{ i}N\|e^{(1)}\|_{2}^{2}\|e^{(2)}\|_{2}\}]\end{equation}
$$\qquad\qquad
=\sup_{x\in [0,1]}\{ \frac{2\theta^{2}}{i} x(1-x)+\frac{i\alpha}{2(1+\alpha)}\ln(1-x)+\frac{i}{2(1+\alpha)}\ln (x)\}-i C_{\alpha}\,.$$
\eqref{b6} thus yields the expected upper bound. To get the lower bound in \eqref{b5}, observe that conditioning by $\|e^{(1)}\|_2$,  the entries of $e^{(1)}$ and $e^{(2)}$ follow uniform laws on the sphere so that Lemma \ref{vec} applies.
Hence, $V_{N}^{\epsilon}$ has probability going to one under this conditionnal measure and we can remove its indicator function in the lower bound of \eqref{b5}. We then apply Laplace method under the Beta law to conclude. Finally,
we see from the above that for any set $A$, any $\delta>0$
$$\Pp^{\theta}(\|e^{(1)}\|_{2}^{2}  \in A)\le \exp\{- N F(\theta, w_{i})+N\delta\}  \int_{A} x^{i L/2}(1-x)^{i M/2}\exp\{  \frac{2\theta^{2}}{ i} Nx(1-x)\} dx$$
from which it follows by Laplace method that the law of $\|e^{(1)}\|_{2}^{2}$ satisfies a large deviation upper bound  with speed $N$ and good rate function which is infinite outside $[0,1]$ and otherwise given by
$$-\frac{2\theta^{2}}{i} x(1-x)-\frac{i\alpha}{2(1+\alpha)}\ln(1-x)-\frac{i}{2(1+\alpha)}\ln x+F(\theta, w_{i})\,.$$
In particular $\|e^{(1)}\|_{2}^{2}$ converges almost surely towards the unique minimizer $x_{\theta,\alpha}$ of this strictly convex function (which vanishes there).

\end{proof}

\section{Identification of the rate function}\label{grf}
To complete the proof of the large deviation upper bound of Theorem \ref{theowldub}, we need to identify 
the rate function, that is prove Proposition \ref{propgrf}. This could a priori be done by saying that the rate function corresponds to the one that is well known for the Gaussian case. But for the sake of completness, we verify directly that we have the same result.

\subsection{Wigner matrices}
We first consider the case of Wigner matrices. Recall that we found for $\beta=1,2$
$$I_{\beta}(x)= \max_{\theta > 0} \left(  J( \sigma, \theta , x) -\frac{\theta^{2}}{\beta} \right) $$
where 

$$
J(\mu,\theta, \la):= \theta v(\theta, \mu, \la) -\frac{\beta}{2} \int \log\left(1+\frac{2}{\beta} \theta v(\theta, \mu, \la) - \frac{2}{\beta}\theta y\right) d\mu( y),
$$
with 
$$
 v(\theta, \mu, \la) := \left\{\begin{array}{ll}
                                 R_\mu(\frac{2}{\beta} \theta), & \textrm{if } 0 \le \frac{2 \theta}{\beta} \le H_{\rm max}(\mu, \lambda) := \lim_{z \downarrow \lambda} \int \frac{1}{z-y} d\mu( y),\\
                                 \lambda - \frac{\beta }{2\theta}, & \textrm{if }\frac{2 \theta }{\beta}>  H_{\rm max}(\mu, \la)\,.
                                \end{array}
                                \right.$$
                                
                                When $\mu=\sigma$, $R_{\sigma}(x)=x$ and $G_{\sigma}(\lambda)=\frac{1}{2}(\lambda-\sqrt{\lambda^{2}-4})$. 
                                
   The critical points of $\varphi(\theta,x)=   J( \sigma, \theta , x) -\frac{\theta^{2}}{\beta}          $  for fixed $x$ satisfy
   $$               \frac{2\theta}{\beta}=\partial_{\theta}J( \sigma, \theta , x)\,.$$
   \begin{itemize}
 \item  For $\frac{2\theta}{\beta}\le G_{\sigma}(x)$, $\varphi (\theta)$ vanishes uniformly as $ J( \sigma, \theta , x)=.\frac{\beta}{2}\int_{0}^{\frac{2}{\beta}\theta}R_{\sigma}(u) du=\frac{\theta^{2}}{\beta}$.
 
 \item For $\frac{2\theta}{\beta}> G_{\sigma}(x)$, the maximum is achieved  at a solution of 
 $$\frac{2\theta_{x}}{\beta}= x-\frac{\beta}{2\theta_x}$$
 which gives 
 $$\frac{2\theta_{x}}{\beta}=
  \frac{1}{2}(x+\sqrt{x^{2}-4})=\frac{1}{G_{\sigma}(x)}\,.$$
 \end{itemize}
 Hence,
 $I_{\beta}(x)=\varphi(\theta_{x},x)$. We can compute its derivative and since $\theta_{x}$ is a critical point of $\varphi$, we find
 $$\partial_{x}I_{\beta}(x)=\partial_{x} \varphi(\theta_{x},x)=\theta_{x}-\frac{\beta}{2} G_{\sigma}(x)=\frac{\beta}{2}\sqrt{x^{2}-4}$$
 which proves the claim since $I_{\beta}(2)=0$.

 \subsection{Wishart matrices}

 Let us now consider Wishart matrices and compute
 $$I_{w_\beta}(x)= \max_{\theta > 0} \left(  J( \sigma_w, \theta , x) -F(\theta, w_\beta) \right) \,.$$
 As in the previous proof we try to compute
 $$\partial_x I_{w_\beta}(x)= \theta_x - \frac{\beta}{2} G_{\sigma_w}(x)$$
 where $\theta_x$ is the argmax of $\varphi(\theta,x)=J( \sigma_w, \theta , x) -F(\theta, w_\beta)$. Note that the latter exists as $\varphi$ is continuous in $\theta$, going to $-\infty$ at infinity. 
 To identify $\theta_x$ we remark that when it is larger than $\frac{\beta}{2}G_{\sigma_w}(x)$, it must satisfy, as a critical point of $\varphi$, 
 $$x=\partial_\theta F(\theta,w_\beta)+\frac{\beta}{2\theta}=:K(\theta)\,.$$
 Our goal is therefore  to identify $K$ and in fact  its inverse. 
 Now, we claim that $\theta\ra F(\theta,w_{\beta})$ is analytic in a neighborhood of $\mathbb R^{+*}$. We recall that it is given in terms of $x_{\theta,\alpha}$, see Lemma \ref{b66}. $x_{\theta,\alpha}$ is a maximizer, and therefore as a critical point it 
  is solution of
 $$\psi(x,\theta)=\frac{1}{\beta^2}\theta^{2}(1-2x) +\frac{1}{(1+\alpha)x}-\frac{\alpha}{(1+\alpha)(1-x)}=0\,.$$
 Clearly $x\ra \psi(x,\theta)$ takes its zeroes away from $0,1$ and  is analytic in a complex neighborhood of $[\epsilon,1-\epsilon]$ for any $\epsilon>0$. Moreover, at $\theta=\infty$, $\psi$ vanishes at $x=1/2$ only. But for $\Re(\theta)>\delta$,  the real part of $-\partial_{x}\psi(\theta,x)$ is bounded below uniformly by some $c(\epsilon)>0$  uniformly  a complex neighborhood $U_{\epsilon}$ of $[\epsilon,1-\epsilon]$ provided the imaginary part of $\theta$ is smaller than some $\kappa_{\epsilon,\delta}>0$. Hence, the implicit function theorem implies that $\theta\ra x_{\theta,\alpha}$, and so $F(.,w_{\beta} )$, is analytic in a complex neighborhood of 
 $\Re(\theta)\ge \delta$. We next show that  for $\theta$ small enough,
 \begin{equation}\label{lj}F(\theta,w_{\beta})=
 \frac{\beta}{2}\int_{0}^{\frac{2}{\beta}\theta}R_{\sigma_{w}}(u) du\,.\end{equation}
 It is clearly lower bounded by this value as for any $M$
 $$F(\theta,w_\beta)\ge\liminf_{N\ra\infty}\frac{1}{N}\ln \mathbb E_{X_N^{(w_\beta)}}[1_{|\lambda_{\rm max}(X_N^{(w_\beta)})|\le M}
 I_{N}(X_N^{(w_\beta)},\theta)]$$
 so that for $\frac{2\theta}{\beta}\le G_{\sigma_{w}}(M)$,\cite[Theorem 1.6]{GuMa05} gives the lower bound. The upper bound is obtained similarly by using the exponential tightness which permits to restrict oneself to $\{|\lambda_{\rm max}|\le M\}$.
 Therefore, we conclude that $K$
  is analytic in $\Re(\theta)>\delta$ and equals $K_{\sigma_w}(\frac{2\theta}{\beta})$ for small $\theta$.  We want to find the inverse of $K$. We thus  look   for an analytic extension of $K_{\sigma_w}$. But in fact $K_{\sigma_w}$ satisfies an algebraic equation. Indeed, 
 observe that $$G_{\sigma_{w}}(x)=2x G_{\pi_{\alpha}}((1+\alpha)x^{2})+\frac{\alpha-1}{(1+\alpha)x}$$
 where it is well known that $G_{\pi_{\alpha}}$, the Stieltjes transform of the Wishart matries,  is solution of 
 $$(2z)^{2}G_{\pi_{\alpha}}(z)^{2}-4z(z+1-\alpha)G_{\pi_{\alpha}}(z)+4z -8\alpha=0\,.$$
 We deduce that at least  for small $x$, $K_{\sigma_w}$ is solution of 
 $$((1+\alpha)K_{\sigma_w} (x) x+1-\alpha)^{2}-2(K_{\sigma_w}(x)+1-\alpha)((1+\alpha)xK_{\sigma_w}(x)+1-\alpha)
 +4(1+\alpha)K_{\sigma_w} (x)^{2} -8\alpha=0\,.$$
 As a consequence, $K$ is also solution of this equation for all $x$, by analyticity. Now, we are looking for the inverse of $K$ and so we deduce that
  $\theta_{x}$ is solution of the equation
 $$(\frac{2}{\beta}(1+\alpha)x\theta_{x}+1-\alpha)^{2}-2(x+1-\alpha)(\frac{2}{\beta}(1+\alpha)x\theta_{x}+1-\alpha)
 +4(1+\alpha)x^{2} -8\alpha=0\,.$$
 For $\frac{2\theta_{x}}{\beta}\le G_{\sigma_{w}}(x)$, the solution is 
$$\frac{2}{\beta}\theta_{x} =\frac{2\alpha}{1+\alpha}\frac{x^{2}+1-\alpha-\sqrt{(x^{2}-1-\alpha)^{2}-4\alpha}}{2x^{2}}
+\frac{1-\alpha}{1+\alpha} \frac{1}{x} =G_{\sigma_{w}}(x)\,.$$
but when $\frac{2\theta_{x}}{\beta}> G_{\sigma_{w}}(x)$
we have to take the other solution of the quadratic equation
$$\frac{2}{\beta}\theta_{x}=\frac{2\alpha}{1+\alpha}\frac{x^{2}+1-\alpha+\sqrt{(x^{2}-1-\alpha)^{2}-4\alpha}}{2x^{2}}
+\frac{1-\alpha}{1+\alpha} \frac{1}{x} \,.
$$
As a result, we then have
$$\partial_x I_{w_\beta}(x)=\theta_{x}-\frac{\beta}{2} G_{\sigma_{w}}(x)=\frac{\beta\alpha}{1+\alpha}\frac{\sqrt{(x^{2}-1-\alpha)^{2}-4\alpha}}{x^{2}} \,,$$
which completes the proof.

\section{ Large deviation lower bounds}
Recall that we need  to prove Lemma  \ref{cruc}, that is  find for 
any $x>2$  (or $\tilde b_{\alpha}$ for Wishart matrices) a  $\theta=\theta_x\ge 0$ such that for any $\delta>0$
and $M$ large enough,

$$\liminf_{N\ra\infty} \frac{1}{N} \ln \frac{   \E[ \mathds{1}_{ X_N^{(\beta)}\in \mathcal A_{x,\delta}^{M}}
I_N(X_N^{(\beta)},\theta)]}{ \E[ I_N(X_N^{(\beta)},\theta)]}\ge 0\,,$$
where we recall that 
$$\mathcal A_{x,\delta}^{M}=\lbrace X: \left|\lambda_{\rm max}(X)-x\right|<\delta\rbrace\cap \{ d(\hat\mu_X^{N}, \sigma_\beta)< N^{-\kappa'}\} \cap \{ \|X\|\le M\}\,. $$

For a  vector   $e$ of the sphere $\mathbb S^{N-1}$ and $X$ a random symmetric matrix, we denote by $\Pp^{(e,\theta)}_N$ the probability measure defined by : 

\[ d \Pp_N^{(e,\theta)}(X) = \frac{ \exp( N \theta \langle X  e,e \rangle) }{ \E_X [\exp( N \theta \langle Xe,e \rangle)]} d \Pp_N(X)\]
where  $\Pp_N$ is the law of $X_{N}^{(\beta)}$. We have
\begin{eqnarray}
\E[ \mathds{1}_{ X_N^{(\beta)}\in \mathcal A_{x,\delta}^{M}}
I_N(X_N^{(\beta)},\theta)]&=&\E_{e}[ \Pp_{N}^{(e,\theta)}( \mathcal A_{x,\delta}^{M}) \E_X [\exp( N \theta \langle Xe,e \rangle)]]\nonumber \\
&\ge& \E_{e}[\mathds \mathds{1}_{ e\in V_N^\epsilon}  \Pp_{N}^{(e,\theta)}( \mathcal A_{x,\delta}^{M}) \E_X [\exp( N \theta \langle Xe,e \rangle)]]\label{tyu}\end{eqnarray}
where we recall that $V_{N}^{\epsilon}=\{e\in \mathbb S^{N-1}:|e_{i}|\le N^{-1/4-\epsilon}\}$. The main point to prove the lower bound will be to show that $\Pp_{N}^{(e,\theta)}( \mathcal A_{x,\delta}^{M}) $ is close to
one for delocalized vectors $e\in V_N^\epsilon$ and then proceed as before to show that $V_N^\epsilon$ has probability close to one under the tilted measure. 
More precisely, we will show that  for $\epsilon \in (\frac{1}{8},\frac{1}{4})$, we can find $\theta$ so that for any $x>2$ (resp $x>\tilde b_{\alpha}$) and $\delta>0$ we can find $\theta_x\ge 0$ so that for $M$ large enough,
\begin{equation}\label{lbw} \lim_{N \to \infty} \inf_{e \in V^\epsilon_N}   \Pp_{N}^{(e,\theta_{x})}( \mathcal A_{x,\delta}^{M})= 1 \,.\end{equation}
This gives the desired estimate  since
we then deduce from \eqref{tyu} that for $N$ large enough so that the above is greater than $1/2$
\begin{eqnarray*}
\E[ \mathds{1}_{ X_N^{(\beta)}\in \mathcal A_{x,\delta}^{M}}
I_N(X_N^{(\beta)},\theta)]
&\ge& \frac{1}{2} \E_{e}[\mathds \mathds{1}_{ e\in V_N^\epsilon}  \E_{X_N^{(\beta)}} [\exp( N \theta \langle X_N^{(\beta)} e,e \rangle)]]\\
\end{eqnarray*}
so that the desired estimate follows from  Proposition \ref{tyui}.  
To prove \eqref{lbw}, the first point is to show that
\begin{lemma}\label{fac} Take $\epsilon\in (0,\frac{1}{4})$. There exists $\kappa >0$ , for $\epsilon>0$, for any $\theta$,
 \begin{itemize}
\item for $K$ large enough:
$$
\lim_{N \to \infty} \sup_{e \in V^\epsilon_N} \Pp^{(e,\theta)}_N\left( \lambda_{\rm max}(X_N^{(\beta)}) \ge K \right)= 0$$
\item

\[ \limsup_{N \to \infty} \sup_{e\in  V^\epsilon_N} \Pp^{(e,\theta)}_N\left( d( \hat \mu^{N}_{X_N^{(\beta)}}, \sigma_\beta) > N^{-\kappa'} \right) = 0 \,.\]

\end{itemize}
\end{lemma} 
\begin{proof}
We hereafter fix a vector $e$ on the sphere. The proof of the exponential tightness is exactly the same as for Lemma \ref{exptight}. Indeed, by Jensen's inequality, we have 
$$ \E_X [\exp( N \theta \langle X_N^{(\beta)} e,e \rangle)]\ge\exp\{ N\theta \E_{X}[ \langle X_N^{(\beta)} e,e \rangle]\}=1$$
Moreover, by Tchebychev's inequality,
 for any $u,v,e\in \mathbb S^{N-1}$, we have
\begin{eqnarray*}
\int \mathds{1}_{\langle X_N^{(\beta)} u,v\rangle\ge K}   \exp( N \theta \langle X_N^{(\beta)}  e,e \rangle) d \Pp_N&\le &
 \exp\{ -N K \}\E_{X}[ \exp( N \theta \langle X_N^{(\beta)}  e,e \rangle+N \langle X_N^{(\beta)} u,v\rangle )]\\
 &\le& \exp\{ -N K \}\exp\{ N\theta^{2}\sum_{i,j}|e_{i}\bar e_{j}+u_{i}\bar v_{j}|^{2}\}\\
 &\le& \exp\{-NK+4\theta^{2}N\}
 \end{eqnarray*}
from which we deduce after taking $u,v$ on a $\delta$-net as in Lemma \ref{exptight} that
$$ \Pp^{(e,\theta)}_N\left( \lambda_{\rm max}(X_N^{(\beta)}) \ge K\right)\le  9^{\beta N} \exp\{ -\frac{1}{4}N K + 4\theta^{2}N \}\,$$
which proves the first point. The second is a direct consequence of Lemma \ref{convmun} and the fact that the log density of $\Pp^{(e,\theta)}_N$ with respect to
$\Pp_N$ is bounded by $\theta N(|\lambda_{\rm max}(X)|+|\lambda_{min}(E)|)$ which is bounded by $\theta KN$ with overwhelming probability by the previous point (recall that $\lambda_{\rm min}(X)$ satisfies the same bounds than $\lambda_{\rm max}(X)$).

\end{proof}

Hence, the main point of the proof is to show that
\begin{lemma}\label{dif} Pick $\epsilon\in ]\frac{1}{8},\frac{1}{4}[$. For any $x>2$ if $\beta=1,2$ and $x>{\tilde b_\alpha}$ if $\beta=w_1,w_2$, there exists $\theta_x$ such that 
for every $\eta >0$, 
\[ \lim_{N \to \infty} \sup_{e \in V^\epsilon_N} \Pp^{(e,\theta_x)}_N[ |\lambda_{\rm max} - x| \geq \eta ] = 0 \]
\end{lemma}

Again, we first consider the simpler Wigner matrix case and then the case of Wishart matrices.

\subsection{Proof of Lemma \ref{dif} for Wigner matrices}

For $e \in V^\epsilon_N$ fixed, let $X^{(e),N}$ be a matrix with  law $\Pp^{(e,\theta)}_N$. 
We have :

\[ X^{(e),N} = \E[ X^{(e),N} ] + (X^{(e),N} - \E[ X^{(e),N} ]) \]
where $\E[X]$ denotes the matrix with entries given by the expectation of the entries of the matrix $X$.
We first show that  $\E[ X^{(e),N} ]$ is approximately a rank one matrix.

\begin{lemma}\label{bv0}
For $\epsilon\in ]\frac{1}{8},\frac{1}{4}[$, there exists $\kappa(\epsilon)>0$ so that for $ e \in V^\epsilon_N$ : 

\[ \E[ X^{(e),N} ] = 2 \theta e \, e^{*} + \Delta^{(e),N} \]
where the spectral radius of $\Delta^{(e),N}$ is bounded by  $N ^{- \kappa(\epsilon)}$ uniformly on $ e \in V^\epsilon_N$. \end{lemma}

\begin{proof}[Proof of the lemma] 

We can express the density of $\Pp_N^{(e,\theta)}$ as the following product : 

\[ \frac{ d \Pp_N^{(e,\theta)} }{d \Pp_{X_N}}(X) = \prod_{i \le  j} \exp( 2^{1_{i\neq j}} \theta \sqrt{N} \Re(e_i \bar{e_j} a_{i,j}) - L_{\mu^N_{i,j}}( 2^{1_{i\neq j}} \theta \sqrt{N} e_i \bar e_j ))\]% \prod_i \exp( \theta \sqrt{N} |e_i|^2 a_{i,i} - L_{\mu_{i,i}^N}(\theta \sqrt{N} |e_i|^2)) \]
where  the $a_{i,j}$ are defined as in the introduction, basically a rescaling of the entries by multiplication by $\sqrt{N}$.

So since we took our $a_{i,j}$ independent (for $i \leq j$), the entries $X^{(e),N}_{i,j}$ remain independent and their mean is given in function of  the Taylor expansion of $L$ as follows :

\[ (\E[ X^{(e),N}) ])_{i,j} = \frac{ L_{\mu^N_{i,j}}'( 2 \sqrt{N} \theta e_i \bar e_j )}{\sqrt{N}} = \frac{2 \theta }{\beta} e_i \bar e_j + \frac{ \delta_{i,j}(2 \sqrt{N} \theta e_i \bar e_j) N \theta^2 |e_i|^2 |e_j|^2}{ \sqrt{N}} \]
if $i \neq j$, and  if $i=j$

\[ (\E[ X^{(e),N} ])_{i,i} = \frac{  L_{\mu_{i,i}^N}'( \sqrt{N} \theta |e_i|^2 )}{\sqrt{N}} =\frac{2 \theta}{\beta} e_i \bar e_i+ \frac{\delta_{i,i}(2 \sqrt{N} \theta |e_i|^2) N \theta^2 |e_i|^4}{ \sqrt{N}} \]
where we used that by centering and variance one,  $L_{\mu^N_{i,j}}'(0) = 0$, $ Hess  L_{\mu^N_{i,j}}(0) = \frac{1}{\beta} Id$ for all $i\neq j,N$,  $L''_{\mu^N_{i,i}}(0)=\frac{2}{\beta}$ for all $i,N$,  and where 
$$ |\delta_{i,j}(t)| \leq  4 \sup_{|u| < t}\max_{i,j,N} \{|L_{\mu^N_{i,j}}^{(3)}(u)|\}\,.$$
Hence, we have
$$ \Delta^{(e),N}_{i,j}={ \delta_{i,j}(2 \sqrt{N} \theta e_i \bar e_j) \sqrt{N} \theta^2 |e_i|^2 |e_j|^2}, 1\le i,j\le N\,.$$
In order to bound the spectral radius of this remainder term, we use the following lemma 
 
\begin{lemma}
Let $A$ be an Hermitian matrix and $B$ a real symmetric matrix  such that : 

\[ \forall i,j, |A_{i,j}| \leq B_{i,j} \]
Then the spectral radius of $A$ is smaller than the spectral radius of $B$.
\end{lemma}
\begin{proof}
Indeed, if we take $u$ on the sphere such that $|| A u ||_2 = ||A||$, then, by denoting $A'$ the matrix $(|A_{i,j}|)$ and $u'$ the vector $(|u_i|)$, we have by the triangular inequality 
\[  ||A|| = ||A u ||_2 \leq || A' u' ||_2 \leq ||B u' ||_2 \leq ||B ||\,. \]
\end{proof}
Therefore, if we choose $C$ so that $C \geq \sup_{N,i,j} \delta_{i,j}(2 \sqrt{N} \theta e_i \bar e_j)  \theta^2$ and set $|e|^2$ to be the vector with entries $(|e_i|^2)_{1\le i\le N}$, we have

\[ ||\Delta^{(e),N} || \leq C \sqrt{N} |||e|^{2}(|e|^{2})^{*}|| \]
Since  $|||e|^{2}(|e|^{2})^{*}|| = |||e|^2||_2^2 = \sum_i e_i^{4} \leq N^{- 4 \epsilon}$, we deduce that
if we take $\epsilon'\in ] 1/8 , 1/4 [$ we have with $\kappa(\epsilon)=1 / 2 - 4 \epsilon$ : 

\[ || \Delta^{(e),N} || =N^{-\kappa(\epsilon)} \,.\]

\end{proof}
\begin{rem} F. Augeri noticed that a maybe more elegant proof of this point would be to use Latala's estimate:
$$\mathbb E[\|Y\|]\le C\sup_{j}\left (\mathbb E \sum_{i}|Y_{i,j}|^{2}\right)^{\frac 1 2}\,.$$
\end{rem}

Now we denote : 

\[ \overline{ X^{(e),N}} := X^{(e),N} - \E[ X^{(e),N} ] \]

The entries of $\overline{ X^{(e),N}}$ are independent, centered of variance $\partial_z\partial_{\bar z} L_{\mu_{i,j}^N}( \theta e_i \bar e_j \sqrt{N}) / N $. Recall that
$\partial_z\partial_{\bar z} L_{\mu_{i,j}^N}( 0 )=1$ and that the third derivative of the Laplace transform of the entries are uniformly bounded 
so that 
\[ \partial_z\partial_{\bar z} L_{\mu_{i,j}^N}( \theta e_i \bar e_j \sqrt{N})  = 1 + \delta_{i,j}( \sqrt{N} |e_ie_j|) = 1+  O(N^{-2\epsilon})\]
uniformly on $V_N^\epsilon$.
We can then consider $\widetilde{ X}^{(e),N}$ defined by :  : 

\[ \widetilde{ X}^{(e),N}_{i,j} = \frac{\overline{ X}^{(e),N}_{i,j}}{ \sqrt{ \partial_z\partial_{\bar z} L_{\mu_{i,j}^N}( \theta e_i \bar e_j \sqrt{N}) }} \]
Set  $Y^{(e), N} = \overline{ X}^{(e),N} - \widetilde{ X}^{(e),N} $. So, we have 

\[ (Y^{(e), N})_{i,j} = \overline{ X}^{(e),N}_{i,j} \left( 1 - \frac{1}{ \sqrt{ \partial_z\partial_{\bar z} L_{\mu_{i,j}^N}( \theta e_i \bar e_j \sqrt{N}) }} \right)\,. \]

We next show that for all $\delta> 0$ : 
\begin{equation}\label{bv} \lim_{N \to + \infty} \sup_{e \in V^\epsilon_N} \Pp[ || Y^{(e), N} || > \delta] =  0 \end{equation}
Indeed, we have the following lemma which is a variant of \cite[Theorem 2.1.22 ]{AGZ} :
\begin{lemma}
Consider for all $N \in \N$ a random Hermitian matrix $A^N$ with independent subdiagonal entries which are centered and for all $k \in \N$ :

\[ r_k^N = \max_{i,j} N^{- k/2} \E[ |A^N_{i,j}|^k]  \]

Suppose that there exists $N_0 \in \N, C > 0$ such that for $N \geq N_0$ : 

\[ r_2^N \leq 1, \qquad  r_k^N \leq k^{Ck} \]
Then for all $\delta > 0$, $\Pp[ \lambda_{\rm max}(A^N) > 2 + \delta]$ goes to zero as $N$ goes to infinity.
\end{lemma}
The proof of this lemma is strictly identical to Theorem 2.1.22 in \cite{AGZ} as we only need to estimate large moments of the matrix, which only requires upper bounds on  moments of the entries (and not equality as assumed in \cite{AGZ}) as soon as the entries are centered.
We apply this  lemma to the matrices $ Y^{(e),N} / \delta$  to derive \eqref{bv}: note that the hypothesis on the upper bound on moments is a clear consequence of our bounds on Laplace transform.

Hence, since
$$X^{(e),N} =\widetilde{X}^{(e),N} + \frac{2\theta}{\beta} e e^{*} + \Delta^{(e),N}+Y^{(e), N}\,,$$
we conclude by combining \eqref{bv} and Lemma \ref{bv0} that
 for  $\epsilon\in ]1/4,1/8[$ and all $\delta >0$

\begin{equation}\label{bv3} \lim_{N \to \infty} \sup_{e \in V^\epsilon_N} \Pp_N^{(e,\theta)}[||X^{(e),N} - ( \widetilde{X}^{(e),N} + \frac{2\theta}{\beta} e e^{*}) ||> \delta ] = 0\end{equation}
since all estimates were clearly uniform on $e\in V^{\epsilon}_{N}$.

And so, to conclude we need only to identify the limit of $\lambda_{\rm max}(\widetilde{X}^{(e),N} + \frac{2\theta}{\beta} e e^{*})$. It is given by the well known BBP transition. We collect below the main elements of the argument for completness.
To identify this limit, we  easily see as in \cite{BGM} that the  eigenvalues of  $\widetilde{X}^{ (e),N} + \frac{2\theta}{\beta} e e^{*}$ satisfy
$$0=\det (z- \widetilde{X}^{(e),N} - \frac{2\theta}{\beta} e e^{*})=\det (z- \widetilde{X}^{(e),N} )\det (1- \frac{2\theta}{\beta} (z- \widetilde{X}^{ (e),N} )^{-1}e e^{*})$$ 
and therefore $z$ is an eigenvalue away from the  spectrum  of $ \widetilde{X}^{(e),N} $ iff

\[ \langle e, (z - \widetilde{X}^{ (e),N})^{-1} e \rangle = \frac{\beta}{2\theta} \,.\]
 But it was shown 
in  Theorem 2.15 of \cite{Erdos} that  for all  $z>2$, all $v\in \mathbb S^{N-1}$, $\langle v, (z - \widetilde{X}^{(e),N})^{-1} v \rangle$ converges almost surely towards $G_\sigma(z)$
and therefore we conclude that the largest eigenvalue $\lambda_{\rm max} ( \widetilde{X}^{ (e),N} + \frac{2\theta}{\beta} e e^{*})$,  must converge towards the solution $\rho_\theta$ to 
$$G_\sigma(\rho_\theta)=\frac{\beta}{2\theta}$$
as soon as it is strictly  greater than $2$.
We find a unique solution to this equation: it is  given by
$$\rho_\theta=\frac{2\theta}{\beta}+\frac{\beta}{2\theta}\,.$$
Reciprocally, for any $x>2$, we can find $\theta_x=\frac{\beta}{2}(x+\sqrt{x^2-4})$ so that $x=\rho_{\theta_x}$. Hence, we have proved that for any sequence of vectors $e\in V_N^\epsilon$ we have the desired estimate for any $\eta>0$
$$\lim_{N \to \infty} \sup_{e \in V^\epsilon_N} \Pp^{(e,\theta_x)}_N[ |\lambda_{\rm max} - x| \geq \eta ] = 0$$
which also entails the convergence of the supremum over $V_N^\epsilon$ and thus the Lemma. 

 \subsection{Proof of Lemma \ref{dif} for Wishart matrices}
We next prove Lemma \ref{dif} for Wishart matrices and fix  $e=(e^{(1)},e^{(2)})\in V_N^\epsilon$.  We decompose as in the previous proof
$$X^{(e),N} =\widetilde{X}^{(e),N} +\mathbb E[ X^{(e),N}]+Y^{(e), N}\,,$$
where  the entries of $\widetilde{X}^{(e),N}$ are centered and with covariance $1/N$ and $Y^{(e), N}$ goes to zero in norm.
We then find by the same argument that 

\[ \E[ X^{(e),N}] = \frac{2 \theta}{i} \begin{pmatrix} 0  & e^{(1)} (e^{(2)})^* \\
e^{(2)} (e^{(1)})^*& 0 \end{pmatrix} +\Delta^{(e),N} \]
where $\|\Delta^{(e),N}\|\le N^{-\kappa(\epsilon)}$ and $e^{(1)}$ (resp. $e^{(2)}$) is the vector made of  the first $L$ (resp. $M$ last) coordinates of $e$. 
Letting 
\[ S^{(e)} = \begin{pmatrix} e^{(1)}  & 0 \\
0 & e^{(2)} \end{pmatrix} \
\mbox{ and 
} T^{(e)} = \begin{pmatrix} 0  & (e^{(2)})^* \\
(e^{(1)})^* & 0 \end{pmatrix} \]
we notice that
\[ 
\begin{pmatrix} 0  & e^{(1)} (e^{(2)})^* \\
e^{(2)} (e^{(1)})^*& 0 \end{pmatrix}= S^{(e)} T^{(e)}\,. \]
Therefore, we need to find $z>{\tilde b_\alpha}$ such that
\begin{equation}\label{mn}
0=\det( z-\widetilde{X}^{N, (e)} -\frac{2\theta}{i} S^{(e)} T^{(e)})= \det( z-\widetilde{X}^{N, (e)})  \det( 1- \frac{2\theta}{i} T^{(e)} (z-\widetilde{X}^{N, (e)})^{-1}S^{(e)})\end{equation}
By writing $R_{\widetilde{X}^{N, (e)}}(z)=(z-\widetilde{X}^{N, (e)})^{-1}$ by blocks with $\widetilde X^{N,(e)}$ with  upper right $L\times M$ block $\widetilde G^{N,(e)}$, we get :

\[ R_{\widetilde{X}^{N, (e)}}(z) = \begin{pmatrix} R_{1,1}(z) & R_{1,2}(z) \\
R_{2,1}(z) & R_{2,2}(z) \end{pmatrix}= \begin{pmatrix} z R_{\widetilde{G}^{N, (e)}({\widetilde{G}^{N, (e)})^*}}(z^2) & \widetilde{G}^{N, (e)} R_{({\widetilde{G}^{N, (e)})^*{\widetilde{G}^{N, (e)}}}}(z^2) \\
 R_{({\widetilde{G}^{N, (e)})^*{\widetilde{G}^{N, (e)}}}}(z^2)({\widetilde{G}^{N, (e)}})^* & z R_{({\widetilde{G}^{N, (e)}})^* {\widetilde{G}^{N, (e)}}}(z^2) 
\end{pmatrix}
\]
where $R_{1,1}$ is $L\times L$, $R_{1,2}$ $L\times M$, $R_{2,2}$ $M\times M$,
we get the simpler equation
\[ \det \left( I -  \frac{2 \theta}{i} \begin{pmatrix} \langle e^{(2)} , R_{2,1}(z) e^{(1)} \rangle & \langle e^{(2)} , R_{2,2}(z) e^{(2)} \rangle \\
\langle e^{(1)} , R_{1,1}(z) e^{(1)} \rangle & \langle e^{(1)} , R_{1,2}(z) e^{(2)} \rangle 
\end{pmatrix}
\right) = 0 
\]
Therefore, we need to find $z$ such that

\begin{equation} \label{total}
 |1 - \frac{2 \theta}{i} \langle e^{(2)} , R_{2,1}(z) e^{(1)} \rangle |^2 - \frac{4 \theta^2}{i^{2}} \langle e^{(2)} , R_{2,2}(z) e^{(2)} \rangle \langle e^{(1)} , R_{1,1}(z) e^{(1)} \rangle =0 \end{equation}
 We are going to prove that
 \begin{lemma}\label{Res}  For any $\delta,\varepsilon >0$
 
\begin{eqnarray*}
\limsup_{N\ra\infty} \sup_{e\in V_N^\epsilon} \mathbb P^{(e,\theta)}_N \left(\sup_{ z \ge \tilde b_\alpha +\varepsilon} | \langle e^{(1)} , R_{1,1}(z) e^{(1)} \rangle - z(1+\alpha)||e^{(1)}||_2^2 G_{MP(\alpha)}((1+\alpha)z^2) |>\delta\right)&=&0\\
\limsup_{N\ra\infty} \sup_{e\in V_N^\epsilon} \mathbb P^{(e,\theta)}_N \left(\sup_{\Im z \ge  \tilde b_\alpha +\varepsilon} | \langle e^{(2)} , R_{2,2}(z) e^{(2)} \rangle - z(1+\alpha)||e^{(2)}||_2^2 G_{MP(1 / \alpha)}((1+\alpha)z^2) | >\delta\right)&=&0\\
\limsup_{N\ra\infty} \sup_{e\in V_N^\epsilon} \mathbb P^{(e,\theta)}_N \left( \sup_{\Im z \ge  \tilde b_\alpha +\varepsilon}| \langle e^{(2)} , R_{2,1}(z) e^{(1)} \rangle | >\delta\right)&=&0\\
\end{eqnarray*}
where $G_{MP(\alpha)}$ is the Stieltjes transform of a Pastur Marchenko law with parameter $\alpha$.
\end{lemma}
We first derive Lemma \ref{dif} assuming that Lemma \ref{Res} holds. 
We have seen in Lemma \ref{b66} that $\|e^{(1)}\|_2$ converges towards $x_{\theta,\alpha}$ almost surely.
Therefore, we arrive to the limiting equation 

\[  (1+\alpha)^2 z^2 G_{MP(\alpha)}((1+\alpha)z^2) G_{MP(1 /\alpha)}((1+\alpha)z^2) = \frac{i^2}{ 4 \theta^2 x_{\theta, \alpha} ( 1 - x_{\theta, \alpha})} \]
Now, we claim that $\varphi(\theta)=  \theta^2 x_{\theta, \alpha} ( 1 - x_{\theta, \alpha})$ is continuous, increasing, going from
$0$ to $+\infty$. As $x_{\theta,\alpha}$ is a complicated solution of $\theta$ ( solution of a degree three polynomial equation), we use the following asymptotic characterization which easily follows from the previous large deviation considerations, see Lemma \ref{b66}:
$$\frac{4\theta}{i} x_{{\theta,\alpha}}(1-x_{\theta,\alpha})=\partial_{\theta}F(\theta,w_{i})\,,$$
where we use that the derivatives of $x_{\theta,\alpha}$ vanishes as it is a critical point of the maximum.
We moreover notice that $G(\theta)=F(\sqrt{\theta},w_{i})$ is convex in $\theta$ (as a supremum of convex functions).
Hence,
$$\varphi (\theta)= \frac{i}{4}\theta \partial_{\theta} F(\theta, w_{i})=\frac{i}{2}\theta^{2} G'(\theta^{2})$$
It follows that $\varphi$ is smooth as $F$ is and moreover
$$\varphi'(\theta)=i(\theta G'(\theta^{2})+\theta^3 G''(\theta))\,.$$
But since $\varphi$ is non negative, $G'$ is non negative and so $\varphi'$ is non negative for all $\theta\ge 0$.
The fact that $\varphi$ goes to infinity at infinity is clear as $x_{\theta,\alpha}$ then goes to $1/2$.
Moreover, for $z > {\tilde b_\alpha}$, $z \mapsto z G_{MP(\alpha)}((1+\alpha)z^2)$ and $z \mapsto z G_{MP(1 / \alpha)}((1+\alpha)z^2)$ are positive and decreasing,
and therefore so are their product.  Hence,  there exist a $\theta_{\alpha}> 0$ so that for every $\theta \geq \theta_{\alpha}$ , the equation above has a unique solution on $[\tilde b_{\alpha}, + \infty[$. Moreover, if we denote $\rho_{\theta}$ this solution, $ \theta \mapsto \rho_{\theta}$ is a bijection from $[\theta_{\alpha}, + \infty[$ onto $[\tilde b_{\alpha}, + \infty[$.

\bigskip

{\it Proof of  Lemma \ref{Res}.} 
We recall that $G=G_{L,M}$ is  a $L\times M$ matrix with centered entries with covariance one and sub-Gaussian tails, $e=(e^{(1)},e^{(2)})$ a unit vector and 
$$R_{1,1}(z)=(z-GG^*)^{-1}, R_{22}(z)=(z-G^*G)^{-1}, R_{1,2}(z)=G(z-G^*G)^{-1}.$$

The first two points of the Lemma are direct consequences of \cite[Theorem 2.5]{Erdos}. 
It remains to see that $\langle e^{(2)} , R_{2,1}(z) e^{(1)} \rangle$  goes to 0 as $N$ goes to infinity. Because $R_{2,1}(z)=G(z-G^*G)^{-1}$ is not the resolvent of the Wishart matrix,  but its multiplication by $G$,
we can not apply directly  \cite[Theorem 2.5]{Erdos}. We will give an elementary proof  of this result,based on classical moment computations.
Indeed, for $\varepsilon>0$, on the set where $\{\|G^* G\|\le b_\alpha+\varepsilon\}$, for $z>b_\alpha+2\varepsilon$ we can expand

\begin{eqnarray*}
\langle e^{(2)} , R_{2,1}(z) e^{(1)} \rangle &=& - \sum \frac{ \langle e^{(1)}, G (G^* G)^k e^{(2)}\rangle}{ z^{ 2k +1}} \\
&=&- \sum_{k=1}^K \frac{ \langle e^{(1)}, G (G^* G)^k e^{(2)}\rangle}{ z^{ 2k +1}}+ O \bigg(\frac{1}{\varepsilon} \left(\frac{b_\alpha+\varepsilon}{b_\alpha+2\varepsilon}\right)^{K+1} \bigg)  \end{eqnarray*}
and hence  it is enough to get the convergence in probability of $K$ moments with $K\ge 2\varepsilon^{-1} \ln \varepsilon^{-1}$ :

\[ \lim_{N \to \infty} \langle e^{(1)}, G (G^* G)^k e^{(2)}\rangle = 0 ,\,  k\le K\,.\]
To this end we first prove that
\begin{equation}\label{exp} \lim_{N \to \infty} \E[ \langle e^{(1)}, G (G^* G)^k e^{(2)}\rangle ] = 0 \end{equation}
and  then
\begin{equation}\label{cov}\lim_{N \to \infty} Var( \langle e^{(1)}, G (G^* G)^k e^{(2)}\rangle ) = 0 \,.\end{equation}
We first prove \eqref{exp}. It is clearly true for $k=0$ by centering of the entries  and so we consider $k\ge 1$.
Let's call $\mathcal{W}_{2k+1}$ the set of words $(v_1,...,v_{2k+2})$ of length $2k+1$ so that $v_{2j} \in \{1,...,L \}$ and $v_{2j+1} \in \{1,...,M \}$. We use the following notation : 

\[ E_v = \E[ a_{v_1,v_2} a_{v_2, v_3} ... a_{v_{2k+1}, v_{2k+2}}] \]
We have 

\[ \E[\langle e^{(1)}, G (G^* G)^k e^{(2)} \rangle] = \frac{1}{ N ^{k+1/2}} \sum_{v \in \mathcal{W} _{2k+1}} e^{(1)}_{v_1} E_v e^{(2)}_{v_{2k+2}} \]

Given a word $v$, we can construct a bipartite graph $G_v$ whose vertices are the $\{v_1, v_3,...\} \cup  \{L + v_2, L+ v_4,...\}$ of whose edges (occasionally multiple) are the $(L + v_{2i}, v_{2i-1})$ and $(L + v_{2i}, v_{2i+1})$. We denote $V^{(1)}(v)$ the number of vertices in $G_v$ lying in $\{ 1,...,L \}$, $V^2(v)$ the number of vertices in $G_v$ lying in $\{ L+1,...,L+M \}$ and $V(v) = V^{(1)}(v)+ V^2(v)$ and $A(v)$ the number of edges of $G_v$. If $e$ is an edge of $G_v$, we denote $n_{v}(e)$ the multiplicity of this edge. 

Let's recall that here the $a_{i,j}$ are independant but not identically distributed. Nevertheless their variance are $1$ and their moments are bounded uniformly i.e. for every $k$ there exists $C_k < + \infty$ such that : 

\[ \sup_{N,i,j} \E[ |a_{i,j}|^k ] \leq C_k \]

For every word $v$ of length $k$, we can define $C_v = \prod_{j \leq k} C_j^{l(v,j)} $ where $l(v,j)$ is the number of edge of multiplicity $j$ in $G_v$. we then have 

\[ |E_v| \leq C_v \]

We say that two words $v,w$ are equivalent if there exists a bijection $\phi: \{ 1,...,L \} \rightarrow \{ 1,...,M \} $ and a bijection $\psi: \{ 1,...,M \} \rightarrow \{ 1,...,M \} $ such that $v_{2j} = \phi( w_{2j})$ and $v_{2j+1} = \psi( w_{2j+1})$. If two words $v$ and $w$ are equivalent then $C_v = C_w$. 

Let $\mathcal{T}_{2k+1}$ be a the quotient set of words of length $2k+1$ for this equivalency relationship. We have 

\begin{eqnarray*}
\E[ \langle e^{(1)}, G (G^* G)^k e^{(2)} \rangle]% & =& \frac{1}{ N ^{k+1/2}} \sum_{v \in \mathcal{W} _{2k+1}} e^{(1)}_{v_1} E_v e^{(2)}_{v_{2k+2}} \\
&=& \frac{1}{ N ^{k+1/2}} \sum_{j =2}^{2k+2} \sum_{ t \in \mathcal{T}_{2k+1}, V(v) = j } \sum_{v | v \sim t} e^{(1)}_{v_1} E_v e^{(2)}_{v_{2k+2}}
\end{eqnarray*} 

Let's notice that if $G_v$ has an edge of multiplicity $1$, then $E_v = 0$ (since the $a_{i,j}$ are independant and centered). So for $E_v$ to be non zero we need that $A(v) \leq (2k+1)/2 $ so $A(v) \leq k$. Since $G_v$ is connected $V(v) \leq A(v)+1 \leq k+1$. If $v \in \mathcal{W}_{2k+1}$, there exists $ N_v := (L - 1)...(L- V^{(1)}(v) +1)( M - 1)...(M - V^{2}(v) +1) \leq N^{V(v) -2}$ equivalent words $w_1$ provided we fix $v_1$ and $v_{2k+2}$ so we have the following bound :

\begin{align*}
 \E [ \langle e^{(1)}, G (G^* G)^k e^{(2)} \rangle] %& =& \frac{1}{ N ^{k+1/2}} \sum_{j =2}^{k+1} \sum_{ t \in \mathcal{T}_{2k+1}, V(v) = j } \sum_{1\leq i \leq L, 1 \leq j \leq M} \sum_{v | v \sim t, v_1 = i, v_{2k+2} =j} e^{(1)}_{i} E_v e^{(2)}_{j} \\
&\leq & \frac{1}{ N ^{k+1/2}} \sum_{j =2}^{k+1} \sum_{ t \in \mathcal{T}_{2k+1}, V(t) = j } C_t N_t \sum_{1\leq v_1 \leq L, 1 \leq v_{2k+2}  \leq M}  |e^{(1)}_{v_1}  e^{(2)}_{v_{2k+2}}|
\end{align*}

By using the Cauchy Schwartz inequality, we have that :

\[ \sum_{1\leq i \leq L, 1 \leq j \leq M} | e^{(1)}_{i}  e^{(2)}_{j}|\leq N \| e^{(1)}||_2\times ||e^{(2)}||_2 \leq N \]
which yields
\begin{eqnarray*}
 \E [\langle e^{(1)}, G (G^* G)^k e^{(2)} \rangle ] &\leq & \frac{1}{ N ^{k-1/2}} \sum_{j =2}^{k+1} \sum_{ t \in \mathcal{T}_{2k+1}, V(t) = j } C_t N^{j -2 } 
\end{eqnarray*}
The leading order term here is in $N^{-1/2}$ for $k\ge 1$ and so $$\lim_{N \to \infty} \sup_{\|e\|_2=1} |\E [\langle e^{(1)}, G (G^* G)^k e^{(2)} \rangle]| =0 \,.$$
We proceed similarly for the covariance \eqref{cov}:
\[ \Var(\langle e^{(1)}, G (G^* G)^k e^{(2)} \rangle) = \frac{1}{ N ^{2k+1}} \sum_{v \in \mathcal{W}_{2k+1}, w \in \mathcal{W}_{2k+1}} e^{(1)}_{v_1} e^{(1)}_{w_1} T_{v,w} e^{(2)}_{v_{2k+2}} e^{(2)}_{w_{2k+2}} \]

Where $ T_{v,w} = E_{v,w} - E_v E_w$ and $E_{v,w} = \E[ a_{v_1,v_2} a_{v_2, v_3} ... a_{v_k, v_{k+1}}a_{w_1,w_2} a_{w_2, w_3} ... a_{w_k, w_{k+1}}]$
We extend naturally the previous definitions to couples of words. Let us  now do the same analysis than before with couples of words. Let's take $\tilde{\mathcal{T}}_{2k+1}$ the quotient set for the equivalency relationship for couples of words. Let $(v,w) \in \tilde{\mathcal{T}}_{2k+1}$

First, if $G_{v,w}$ is not connected, since it is the union of two connected graphs $G_v$ and $G_w$, we have that $G_v$ and $G_w$ don't have any edges in common and so, by independence of the entries $T_{v,w} = 0$. So we can assume that $G_{v,w}$ is connected.

Then several cases arise : 

First, if $v_1 \neq w_1$ and $v_{2k+2} \neq w_{2k +2}$, then if one edge of $G_{v,w}$ is of multiplicity 1, then $T_{v,w} =0$. 
So we can assume that all edges are of multiplicity at least $2$. We deduce that $A(v,w) \leq 2k +1$ and $V(v,w)\leq 2k +2$. Let $N_{v,w}$ be the number of couple of words equivalent to $(v,w)$ provided $(v_1,w_1,v_{2k+2}, w_{2k +2})$ are fixed, we have $N_{v,w} \leq N^{2k - 2}$. Hence

\[ \sum_{(u,t) \sim (v,w)} e^{(1)}_{u_1} e^{(1)}_{t_1} T_{v,w} e^{(2)}_{u_{2k+2}} e^{(2)}_{t_{2k+2}} \leq N^{2k} (C_{v,w} - C_v C_w) \]

Then, if $v_1 = w_1$ and $v_{2k+2} \neq w_{2k +2}$ or if $v_1 \neq w_1$ and $v_{2k+2} = w_{2k +2}$, the same reasoning concerning the edges holds. So, we have $V(v,w)\leq 2k +2$ and if $N_{v,w}$ is the number of couple of words equivalent to $(v,w)$ provided $(v_1,w_1,v_{2k+2}, w_{2k +2})$ are fixed, we have $N_{v,w} \leq N^{2k - 1}$. If we are in the case $v_1= w_1$ : 

\[ \sum_{(u,t) \sim (v,w)} e^{(1)}_{u_1} e^{(1)}_{t_1} T_{v,w} e^{(2)}_{u_{2k+2}} e^{(2)}_{t_{2k+2}} \leq N^{2k} ||e^{(1)}||^2 (C_{v,w} - C_v C_w)  \]

And lastly, $v_1 = w_1$ and $v_{2k+2} = w_{2k +2}$ we have again $N_{v,w} \leq N^{2k}$ and 

\[  \sum_{(u,t) \sim (v,w)} e^{(1)}_{u_1} e^{(1)}_{t_1} T_{v,w} e^{(2)}_{u_{2k+2}} e^{(2)}_{t_{2k+2}} \leq N^{2k} ||e^{(1)}||^2 ||e^{(2)}||^2 (C_{v,w} - C_v C_w) \]

So we have 

\[ \Var(\langle e^{(1)}, G (G^* G)^k e^{(2)} \rangle) = O \left( \frac{1}{N} \right) \]

\hfill\qed

\section{Appendix: Proof of Lemma \ref{convmun}}

In this section, we want to prove that the assumptions \ref{A0} and \ref{A0c} are verified if $\mu_{i,j}$ are  supported inside a common compact $K$ or satisfy a log-Sobolev inequality with a uniformly bounded constant $c$
for the  matrices $X_N^{(1)},X_N^{(2)},X_N^{(w_{1})},X_{N}^{(w_{2})}$. 

\begin{lemma} There exists $\kappa\in (0,\frac{1}{10})$ such that

\[ \lim_{N \to \infty} \frac{1}{N} \ln \Pp[ d( \mu_{X_N^{(\beta)}}, \sigma_\beta) > N^{-\kappa} ] = - \infty \]
for $\beta=1,2,{w_{1}},w_{2}$, where $\sigma_\beta$ is the semi-circle law when $\beta=1,2$ and the Pastur-Marchenko law with index $\alpha$ if $\beta={w}$ (in the latter case we assume $M/N-\alpha=o(N^{-\kappa}))$.
\end{lemma}

For this, we will use two concentration results respectively from \cite{GZ} and \cite{B}. 

\begin{theo} By    \cite[Theorem 1.4)] {GZ} (for the compact case) and   \cite[Corollary 1.4 b)] {GZ} (for the logarithmic Sobolev case), we have for  $\beta=1,2,{w_{1}},w_{2}$,  and for  $N$ large enough
$$ \limsup_{N\ra\infty}\frac{1}{N^{7/6}}\ln  \Pp[ d( \mu_{X_N^{(\beta)}}, \E[ \mu_{X_N^{(\beta)}} ]) > N^{-1/6}]  <0$$
where $d$ is the Dudley distance.
\end{theo}
We therefore only need to show that

\begin{theo}( \cite[Theorem 4.1]{B}) 
If we let for every $N$ :

\[ F_{X_N^{(1)}} (x) = \mu_{X_N^{(1)}}( ] - \infty, x]) \]
\[ F_{\sigma_1}(x) = \sigma_1( ] - \infty, x] ) \]

we have that 

\[ \sup_{x \in\R} | F_{\sigma_1}(x) - \E[F_{X_N^{(1)}}(x)]| = O( N ^{-1/4})\,.\]
\end{theo}

In order to conclude, we need only to use Lemma \ref{exptight} to see that $F_{X_N^{(1)}}(-M)$ and $1-F_{X_N^{(1)}}(M)$ decay exponentially fast in $N$ for some fixed $M$
so that
$$ d(\E[ \mu_{X_N^{(1)}} ],\sigma_1)\le 4e^{-N}\|f\|_\infty + 2M \|f\|_L \sup_{x \in\R} | F(x) - \E[F_{X_N^{(1)}}(x)]| =o(N^{-1/6})\,.$$
The same results hold in the complex case. For Wishart matrices, we rely on \cite[Theorem {w}.1 and {w}.2]{B2}. Recall that $W_N=G_{L,M} (G_{L,M})^*$.

\begin{theo}( \cite[Theorem 4.1]{B}) Assume that $M/N\in (1, \epsilon^{-1})$ for some fixed $\epsilon$ and $M/N$ converges towards $\alpha$. Then
\[ \sup_{x \in\R} | F_{\pi_\alpha}(x) - \E[F_{W_N}(x)]| = O( N ^{-1/10})\,.\]
\end{theo}
We can as well use Lemma \ref{exptight} to conclude that $1-F_{W_N}(M)$ goes to zero like $e^{-N}$ for $M$ large enough.
Finally, we conclude by noticing that since
$$\int f(x) d\E[\hat\mu_{X_N^{w}}](x)=\frac{N}{N+M} \int (f(\sqrt{\lambda})+f(-\sqrt{\lambda}))d\hat\mu_{W_N}(\lambda) +\frac{M-N}{N} f(0),$$
we have
\begin{eqnarray*}
\left|\int f(x) d(\E[\hat\mu_{X_N^{w}}]-\sigma_{w})(x)\right|&\le & \|f\|_\infty (|\frac{M}{N}-\alpha| +e^{-N}) +\int_0^M |\partial_\lambda f(\sqrt{\lambda})|  | F_{\pi_\alpha}(\lambda) - \E[F_{W_N}(\lambda)]|d\lambda\\
&\le &\|f\|_L (N^{-\kappa}+e^{-N} + 2 MN^{-\frac{1}{10}})\\
\end{eqnarray*}

\bibliographystyle{plain}
\bibliography{biblioRad}

\end{document}